\documentclass[a4paper]{article}

\usepackage{tikz-cd}
\usetikzlibrary{cd}
\usepackage[arrow, matrix, curve]{xy}

\usepackage{amssymb}
\usepackage{amsmath}
\usepackage{amsthm} 
\usepackage{mathtools}

\usepackage{hyperref}
\hypersetup{
    colorlinks=true,
    linkcolor=blue,
    filecolor=magenta,      
    urlcolor=cyan,
}

\usepackage[utf8]{inputenc}
\usepackage{hyperref}
\hypersetup{
    colorlinks=true,
    linkcolor=blue,
    filecolor=magenta,      
    urlcolor=cyan,
}
\usepackage{amsmath, amssymb, amsthm, bbm, graphicx, cleveref, caption, tikz, enumerate, todonotes,comment}
\usepackage{subcaption}
\usepackage{wrapfig}

\usepackage{geometry}
\usepackage{mathtools}
\usepackage{pdfpages}

\DeclarePairedDelimiter\abs{\lvert}{\rvert}

\DeclareMathOperator{\Id}{Id}

\DeclareMathOperator{\Mat}{{Mat}}
\DeclareMathOperator{\Span}{{span}}
\DeclareMathOperator{\diag}{{diag}}

\DeclareMathOperator{\GL}{GL}

\DeclareMathOperator{\U}{U}
\DeclareMathOperator{\SU}{SU}
\DeclareMathOperator{\Orth}{O}
\DeclareMathOperator{\SO}{SO}

\newtheorem{theorem}{Theorem}
\newtheorem{lemma}{Lemma}

\newtheorem{definition}{Definition}
\newtheorem{remark}{Remark}

\newcommand{\R}{\mathbb{R}}

\graphicspath{{figures/}}
\newcommand{\N}{\mathbb{N}}

\newcommand{\Z}{\mathbb{Z}}
\newcommand{\p}{\partial}
\renewcommand{\epsilon}{\varepsilon}
\newcommand{\dx}{\: \mathrm{d}}

\newcommand{\nm}{\noalign{\smallskip}}
\newcommand{\ds}{\displaystyle}

\title{Topological phenomena in honeycomb Floquet metamaterials\thanks{\footnotesize
This work was supported in part by the Swiss National Science Foundation grant number
200021--200307.}}

\author{Habib Ammari\thanks{\footnotesize Department of Mathematics, ETH Z\"urich, R\"amistrasse 101, CH-8092 Z\"urich, Switzerland (habib.ammari@math.ethz.ch, thea.kosche@sam.math.ethz.ch).}  \and Thea Kosche\footnotemark[2]}

\date{}

\begin{document}
	\maketitle
	
\begin{abstract}
    Being driven by the goal of finding edge modes and of explaining the occurrence of edge modes in the case of time-modulated metamaterials in the high-contrast and subwavelength regime, we analyse the topological properties of Floquet normal forms of periodically parameterized time-periodic linear ordinary differential equations $\left\{\frac{d}{dt}X = A_\alpha(t)X\right\}_{\alpha \in \mathbb{T}^d}$. In fact, our main goal being the question whether an analogous principle as the bulk-boundary correspondence of solid-state physics is possible in the case of Floquet metamaterials, i.e., subwavelength high-contrast time-modulated metamaterials. This paper is a first step in that direction. Since the bulk-boundary correspondence states that topological properties of the bulk materials characterize the occurrence of edge modes, we dedicate this paper to the topological analysis of subwavelength solutions in Floquet metamaterials. This work should thus be considered as a basis for further investigation on whether topological properties of the bulk materials are linked to the occurrence of edge modes. The subwavelength solutions being described by a periodically parameterized time-periodic linear ordinary differential equation $\left\{\frac{d}{dt}X = A_\alpha(t)X\right\}_{\alpha \in \mathbb{T}^d}$, we put ourselves in the general setting of periodically parameterized time-periodic linear ordinary differential equations and introduce a way to (topologically) classify a Floquet normal form $F,~P$ of the associated fundamental solution $\left\{X_\alpha(t) = P(\alpha,t)\exp(tF_\alpha)\right\}_{\alpha \in \mathbb{T}^d}$. This is achieved by analysing the topological properties of the eigenvalues and eigenvectors of the monodromy matrix $X_\alpha(T)$ and the Lyapunov transformation $P(\alpha,t)$. The corresponding topological invariants can then be applied to the setting of Floquet metamaterials. In this paper these general results are considered in the case of a hexagonal structure. We provide two interesting examples of topologically non-trivial time-modulated hexagonal structures.
\end{abstract}

\noindent{\textbf{Mathematics Subject Classification (MSC2000):} 35J05, 35C20, 35P20, 74J20
		
\vspace{0.2cm}
		
	\vspace{0.5cm}
		
\tableofcontents

\section{Motivation and introduction}

    The study of topological properties  in periodic
    physical systems is one of the most active  fields
    in solid-state physics. The topological invariants are the building elements of topological band theory since they  imply the presence of non-trivial bulk
    topologies, giving rise to topologically protected edge modes \cite{shortcourse,zak_experiment,phases,kanereview,bulkbdy,SSH,top10,zak}.
    Topological properties of electronic structures have been mathematically studied in the setting of the Schrödinger operator \cite{drouot2, drouot1,fefferman,fefferman2,fefferman3,fefferman4, lee-thorp}. 
    Topological invariants have been defined to capture the crystal's  properties. Then, if part of a crystalline structure is replaced with an arrangement that is associated with a different value of this invariant, not only will certain frequencies be localized to the interface  but this behaviour will be stable with respect to imperfections. These eigenmodes are known as \emph{edge modes} and we say that they are \emph{topologically protected} to refer to their robustness. This principle is known as the \emph{bulk-boundary correspondence} in quantum settings \cite{drouot2, drouot4,graf2013bulk2d,graf2018bulk,graf2018bulk2d,bulkbdy,prodan}. Here, the term \emph{bulk} is used to refer to parts of a crystal that are away from an edge.
        
    Taking inspiration from quantum mechanics, subwavelength topological photonic and phononic crystals, based on locally resonant crystalline structures with large material contrasts, have been studied both numerically, experimentally, and mathematically in \cite{ammari2020robust,ammari2020topological,zak_acoustic,top_acoustic_SSH,hai,
    top_review,pocock2018,top_subwavelength,top_subwavelength2,Yves2,Yves1,Yves3,zhao2018}.
    Subwavelength crystals (called also metamaterial structures) allow for the manipulation and localization of waves on very small spatial scales (much smaller than the wavelength) and are therefore very useful in physical applications, especially situations where the operating wavelengths are very large. Recently, this field  has experienced tremendous advances by exploring the novel and promising area of time-modulations. Research on {\em Floquet metamaterials} (or time-modulated metamaterials) aims to explore new phenomena arising from the temporal modulation of the material parameters of the structure.  It has enabled to open new paradigms for the manipulation of wave-matter interactions in both spatial and temporal domains  \cite{ammari2020time,fleury2016floquet,floq2,floq3,koutserimpas2018zero,
    koutserimpas2020electromagnetic,refedge1,rechtsman2013photonic,
    wilson2019temporally,floq2,wilson2018temporal}. To the best of our knowledge,  the mathematical analysis of wave propagation properties of Floquet metamaterials has been just started.  In \cite{ammari2020time}, a discrete characterization of the
    band structure in Floquet metamaterials is introduced. This characterization provides both theoretical insight and efficient numerical
    methods to compute the dispersion relationship of time-dependent structures.
    A study of exceptional points in the case of Floquet metamaterials is conducted in 
    \cite{thea2022jde}. Furthermore, in \cite{jinghao} the possibility of achieving non-reciprocal wave propagation in Floquet metamaterials is proven. 

    We are interested in edge modes, particularly edge modes which appear in honeycomb (e.g. graphene-like) structures and which result from topological non-trivialities due to time-modulation of the material parameters. In  \cite{ammari2020time}, it is shown that the subwavelength solutions to periodically time-modulated metamaterials are described by a second order linear ordinary differential equation (lODE). Using this result in the case of infinite periodic  crystals, it follows that the subwavelength solutions are described by a periodically parameterized lODE $\left\{\frac{d}{dt}X = A_\alpha(t)X\right\}_{\alpha \in \mathbb{T}^d}$ with time-periodic coefficient matrix $A_\alpha(t)$ and parameter space $\mathbb{T}^d \cong \mathbb{S}^1 \times \ldots \times \mathbb{S}^1$ ($d$-times) denoting the $d$-dimensional torus. This motivated the study of general periodically parameterized periodic lODEs of the form 
            \begin{align}\label{eq:periodic_param_lODE}
                \left\{\frac{d}{dt}X = A_\alpha(t)X\right\}_{\alpha \in \mathbb{T}^d},
            \end{align}
        where $A_\alpha$ is periodic with respect to time $t$, say with a period $T$.
    
   In this paper,  we will first derive topological invariants which distinguish different topological properties in the case of a general periodically parameterized periodic lODE, see section \ref{sec:Topological_phenomena_of_periodically_parameterized_linear_ODEs}. To this end, the existence of continuously parameterized Floquet normal form is proven in subsection \ref{subsec:Continuously_differentiably_parameterized_Floquet_normal_form}. A continuously parameterized Floquet normal form associated to a given periodically parameterized periodic lODE then serves as a basis for the definition of the introduced topological invariants. In fact, we will distinguish between two types of topological invariants, to which we will refer to as Type I and Type II. The reason being that the Type I topological invariants are also meaningful in the setting of continuously parameterized \emph{constant} lODEs, whereas non-trivial Type II topological invariants arise solely in the setting of time-dependent continuously parameterized periodic lODEs. The topological invariants are derived in subsection \ref{subsec:Derivation_of_topological_invariants}.

    In section \ref{sec:Application_to_high-contrast_acoustic_hexagonal_metamaterial_structures}, these results are then applied to graphene-like Floquet metamaterials. First the precise setting from \cite{ammari2020time} is recalled in subsection \ref{subsec:metamaterial_setting}. Due to the symmetry of honeycomb structures the parameter space will be reduced from a two-dimensional torus $\mathbb{T}^2$ to a one-dimensional torus, that is to the circle $\mathbb{S}^1$. The low dimension of the parameter space restricts the variety of meaningful topological invariants. The respective results are presented in subsection \ref{subsec:S1_param_top_invar}. Finally the results will be applied to honeycomb Floquet metamaterials and two interesting examples will be presented in subsection \ref{subsec:TypeI.a_top_eff}.


\section{Topological phenomena of periodically parameterized linear ODEs}\label{sec:Topological_phenomena_of_periodically_parameterized_linear_ODEs}

    The goal of this section is to investigate some topological phenomena arising in the case of periodically parameterized (with respect to $\alpha$) periodic lODEs $\{\frac{d}{dt}X = A(\alpha,t)X\}_{\alpha \in \mathbb{T}^d}$, where $A(\alpha,t)$ is periodic with respect to $t$ with period $T$ which is independent of $\alpha \in \mathbb{T}$. This will be achieved through an analysis of the Floquet normal form associated to the corresponding parameterized fundamental solution $ \{X_\alpha(t)\}_{\alpha \in \mathbb{T}^d}$ and where $\mathbb{T}^d$ is regarded as $\mathbb{R}^d/L$, where $L \subset \mathbb{R}^d$ is a full-dimensional lattice. To this end, we will first prove the existence of a continuously differentiable Floquet normal form associated to the parameterized lODE $\{\frac{d}{dt}X = A(\pi(\gamma),t)X\}_{\gamma \in \mathbb{R}^d}$, where $\pi$ is the canonical projection $\pi:\mathbb{R}^d \rightarrow \mathbb{R}^d/L = \mathbb{T}^d$. Under some assumptions on $A$ and since $\mathbb{R}^d$ is simply connected, it will be possible to construct a continuously differentiable Floquet normal form $X_{\pi(\gamma)} = P(\gamma,t)\exp(tF_\gamma)$ associated to the parameterized fundamental solution $X_{\pi(\gamma)}(t)$. Analysis of the $\mathbb{R}^d$-parameterized Floquet normal form will lead to the definition of two topological invariants which will be called \emph{Type I} and \emph{Type II topological invariants}.


    \subsection{Setting}\label{sec:Setting}

    Let $A: \mathbb{R}^d\times\mathbb{R} \rightarrow \Mat_{N\times N}(\mathbb{C}), (\gamma, t) \rightarrow A(\gamma,t)$ be a continuously differentiable function, which is periodic with respect to $(\gamma,t)\in \mathbb{R}^{d+1}$. That is, there is a maximal lattice $L \subset \mathbb{R}^{d}$ and a positive number $T>0$, such that translations by lattice vectors of $L \times m\mathbb{Z}$ leave $A$ invariant. Denoting by $\mathbb{T}^d$ the $d$-dimensional torus given by $\mathbb{R}^d/L$ and by $\mathbb{S}^1$ the circle $\mathbb{R}/T\mathbb{Z}$, the function $A$ can be factored in the following way
        \begin{align*}
            \begin{xy}
                \xymatrix{
                    \mathbb{R}^{d}\times \mathbb{R} \ar[rr]^{\pi \times \psi} \ar[drr]_{A(\gamma,t)}  && \mathbb{T}^{d} \times \mathbb{S}^1 \ar[d]^{\tilde{A}(\alpha,\overline{t})}\\
                                                                                    && \Mat_{N\times N}(\mathbb{C}),
                }
            \end{xy}
        \end{align*}
    where $\pi: \mathbb{R}^{d} \rightarrow \mathbb{T}^{d}$ is given by $\pi(x) = x\mod L$ and $\psi: \mathbb{R} \rightarrow \mathbb{S}^{1}$ is given by $\psi(t) = t \mod T$. In what follows, the map $A: \mathbb{R}^d\times\mathbb{R} \rightarrow \Mat_{N\times N}(\mathbb{C})$ and the associated map $ \tilde{A}: \mathbb{T}^{d} \times \mathbb{S}^1 \rightarrow \Mat_{N\times N}(\mathbb{C})$ will both be denoted by $A$. It will be distinguished between $A$ and $\tilde{A}$ by using different notation for the argument, that is, $A(\gamma,t)$ will indeed denote $A(\gamma,t)$ where $A$ is given by $A: \mathbb{R}^d\times\mathbb{R} \rightarrow \Mat_{N\times N}(\mathbb{C})$ with $\gamma \in \mathbb{R}^d$ and $t \in \mathbb{R}$, $A(\alpha,\overline{t})$ will denote $\tilde{A}(\alpha, \overline{t})$ with $\alpha \in \mathbb{T}^d$, $\overline{t} \in \mathbb{S}^1$ and $A(\alpha,t)$ will denote $\tilde{A}(\alpha, \overline{t})$ where $\overline{t} = t \mod T$ with $\alpha \in \mathbb{T}^d$ and $t \in \mathbb{R}$.

    In the following the topological invariants associated to the parameterized fundamental solution of the parameterized lODE
        \begin{align}\label{eq:lODE}
            \left\{\frac{d}{dt}X = A(\alpha,t)X\right\}_{\alpha \in \mathbb{T}^d}
        \end{align}
    will be defined and investigated. To this end, the family of fundamental solutions $\{X_\alpha\}_{\alpha \in \mathbb{T}^d}$ to the parameterized lODE \eqref{eq:lODE} will be understood as a function
        \begin{align*}
            X : & \mathbb{T}^d\times\mathbb{R} &\longrightarrow  &\GL_{N}(\mathbb{C})\\
                & (\alpha,t)                   &\longmapsto      &X_\alpha(t),
        \end{align*}
    where for each $\alpha \in \mathbb{T}^d$ and each $t \in \mathbb{R}$ the matrix $X_\alpha(t)$ denotes the fundamental solution of the lODE $\frac{d}{dt}X_\alpha = A(\alpha,t)X_\alpha$ evaluated at $t \in \mathbb{R}$. In the following it will also be assumed that $X_{\alpha}(T)$ is diagonalizable for all $\alpha \in \mathbb{T}^{d}$ and such that for each $\alpha \in \mathbb{T}^d$ there exists a neighborhood $W \subset \mathbb{T}^d$ of $\alpha$ and continuously differentiable functions $\lambda_1,\ldots,\lambda_N: W \rightarrow \mathbb{C}$ and $\eta_1,\ldots,\eta_N: W \rightarrow \mathbb{CP}^{N-1}$, such that for all $\alpha' \in W$ the $N$-tuple $(\lambda_1(\alpha'),\ldots,\lambda_N(\alpha'))$ gives the eigenvalues of $X_{\alpha'}(T)$ and $(\eta_1(\alpha'),\ldots,\eta_N(\alpha'))$ are the respective eigenspaces. We will call this property \emph{local $C^1$-diagonalizability}.\footnote{In order to avoid pathologies in the case of eigenvalues with geometric multiplicity greater than $1$, one should further assume that $\Span_\mathbb{C}(\eta_1, \ldots, \eta_N) = \mathbb{C}^N$, that is, that the eigenspaces $(\eta_1, \ldots, \eta_N)$ span the whole space $\mathbb{C}^N$.}
    

    \subsection{Continuously differentiably parameterized Floquet normal form}\label{subsec:Continuously_differentiably_parameterized_Floquet_normal_form}

    In the following subsection, the existence of a continuously differentiably parameterized Floquet normal form associated to $C^1$-diagonalizable fundamental solutions of a parameterized periodic lODE 
        \begin{align}\label{eq:Rd_param_period_lODE}
            \left\{\frac{d}{dt}X = A(\gamma,t)X\right\}_{\gamma \in \mathbb{R}^d}
        \end{align}
    will be proven. This existence will then be used in the subsequent sections to analyse topological phenomena arising in the case of periodically parameterized periodic lODEs.
    
    Given a periodic linear ordinary differential equation
        \begin{align}\label{eq:non_parm_lODE}
            \frac{d}{dt}X = A(t)X
        \end{align}
    with $A: \mathbb{R} \rightarrow \Mat_{N \times N}(\mathbb{R})$ being a $T$-periodic function, denote by $X$ its associated fundamental solution, that is, the solution to the initial value problem    
        \begin{align*}
            \frac{d}{dt}X   &= A(t)X,    \\
            X(0)            &= \Id_N.
        \end{align*}
    Then Floquet theory states that $X$ can be decomposed into a \emph{Floquet normal form} which is a decomposition of the form
        \begin{align*}
            X(t) = P(t)\exp(tF),
        \end{align*}
    where $P: \mathbb{R} \rightarrow \GL_N(\mathbb{C})$ is $T$-periodic and $F \in \Mat_{N \times N}(\mathbb{C})$ is a constant matrix. Note that `the' Floquet normal form associated to the fundamental solution of $X$ is never well-defined due to the non-injectivity of the exponential map. In fact, the eigenvalues of $F$ are uniquely defined up to adding an integer multiple of $2\pi i/T$.
    
    In the following we will refer to any choice of $F$ as a \emph{Floquet exponent matrix} associated to $X$ or the lODE \eqref{eq:non_parm_lODE} and to the corresponding choice of $P$ as the associated \emph{Lyapunov transform}.
    
    This choice of naming is due to the fact that a Lyapunov transform $P$ associated to the lODE \eqref{eq:non_parm_lODE} can serve to transform the time-dependent lODE \eqref{eq:non_parm_lODE} into a lODE with constant coefficients. In fact, making the substitution $Y(t) = P(t)^{-1}X(t)$ will transform the time-dependent lODE \eqref{eq:non_parm_lODE} into the following lODE:
        \begin{align*}
            \frac{d}{dt}Y = F Y,
        \end{align*}
    where $F$ is the Floquet exponent matrix corresponding to $P$.
    
    One possibility to obtain a choice of Floquet normal form (even all choices of Floquet normal forms) associated to the periodic lODE \eqref{eq:non_parm_lODE} is to diagonalize the \emph{monodromy matrix} $X(T)$ of the lODE \eqref{eq:non_parm_lODE}:
        \begin{align*}
            X(T) = \eta\Lambda\eta^{-1} = \begin{pmatrix}
                | & &| \\
                \eta_1 &\ldots &\eta_N \\
                | & &| \\
            \end{pmatrix}\begin{pmatrix}
                \lambda_1 & & \\
                 &\ddots & \\
                & &\lambda_N \\
            \end{pmatrix}\begin{pmatrix}
                | & &| \\
                \eta_1 &\ldots &\eta_N \\
                | & &| \\
            \end{pmatrix}^{-1}.
        \end{align*}
    The eigenvalues $(\lambda_1,\ldots,\lambda_N)$ are than the \emph{characteristic multipliers} of the system \eqref{eq:non_parm_lODE} and all possible choices of \emph{Floquet exponents} are in bijection with the solutions $(\mu_1,\ldots,\mu_N)$ to the equation
        \begin{align*}
            \begin{pmatrix}
                \lambda_1 & & \\
                 &\ddots & \\
                & &\lambda_N \\
            \end{pmatrix} = \begin{pmatrix}
                \exp(T\mu_1) & & \\
                 &\ddots & \\
                & &\exp(T\mu_N) \\
            \end{pmatrix}.
        \end{align*}
    Given a choice of Floquet exponents $(\mu_1,\ldots,\mu_N)$ the corresponding Floquet exponent matrix can be obtained via    
        \begin{align*}
            F = \eta\mu\eta^{-1} = \begin{pmatrix}
                | & &| \\
                \eta_1 &\ldots &\eta_N \\
                | & &| \\
            \end{pmatrix}\begin{pmatrix}
                \mu_1 & & \\
                 &\ddots & \\
                & &\mu_N \\
            \end{pmatrix}\begin{pmatrix}
                | & &| \\
                \eta_1 &\ldots &\eta_N \\
                | & &| \\
            \end{pmatrix}^{-1}
        \end{align*}

    Any choice of $F$ and in particular the above choice of $F$ uniquely determines the Lyapunov transform $P$ as
        \begin{align*}
            P(t) = X(t)\exp(-tF),
        \end{align*}
    obtaining one possible choice of Floquet normal form associated to the periodic lODE \eqref{eq:non_parm_lODE}.

    In fact, one obtains all possible Floquet normal forms via the above procedure. Indeed, it is easy to see that the family of possible Floquet normal forms is in bijection with
        \begin{align*}
            \left\{\begin{pmatrix}
                \mu_1 + n_1\frac{2\pi i}{T} & & \\
                &\ddots & \\
               & &\mu_N + n_N\frac{2\pi i}{T} \\
            \end{pmatrix} \quad : \qquad n_1,\ldots, n_N \in \mathbb{Z}
             \right\}.
        \end{align*}

    We are interested in continuously differentiably parameterized Floquet normal forms
        \begin{align*}
            X(\gamma,t) = P(\gamma,t)\exp(tF_\gamma)
        \end{align*}
    associated to parameterized periodic lODEs as in equation \eqref{eq:Rd_param_period_lODE}.
    To this end, we will use the above procedure which insures that any continuously differentiable parameterization of a choice of Floquet exponents $(\mu_1(\gamma),\ldots,\mu_N(\gamma))$ and any continuously differentiable parameterization of a choice of eigenspaces $(\eta_1(\gamma),\ldots,\eta_N(\gamma))$ provides a continuously differentiable parameterization of a choice of Floquet normal form of the system \eqref{eq:Rd_param_period_lODE}.
    
    Given a parameterized time-periodic lODE of the form (as in equation \eqref{eq:Rd_param_period_lODE})
        \begin{align*}
            \left\{\frac{d}{dt}X = A(\gamma,t)\right\}_{\gamma \in \mathbb{R}^d},
        \end{align*}
    where $A(\gamma,\cdot): \mathbb{R} \rightarrow \Mat_{N \times N}(\mathbb{C}^d)$ is $T$-periodic for all $\gamma \in \mathbb{R}^d$, we will thus need to construct continuously differentiable parameterizations of the Floquet exponents and corresponding eigenspaces of the Monodromy matrix.
    To this end, the following definition will become useful and is in fact, necessary for the construction of the desired continuously differentiable parametrizations.

    \begin{definition}[local $C^1$-diagonalizability]
        Let $\mathcal{M}$ be a differentiable manifold and let $M:\mathcal{M} \rightarrow \Mat_{N \times N}(\mathbb{C})$ be a matrix-valued function. Then $M$ is \emph{locally $C^1$-diagonalizable} if the following holds. For all $\gamma \in \mathcal{M}$ there exists a neighborhood $W \subset \mathcal{M}$ of $\gamma$ and continuously differentiable functions $\lambda_1,\ldots,\lambda_N: W \rightarrow \mathbb{C}$ and $\eta_1,\ldots,\eta_N: W \rightarrow \mathbb{CP}^{N-1}$, such that for all $\gamma' \in W$ the $N$-tuple $(\lambda_1(\gamma'),\ldots,\lambda_N(\gamma'))$ gives the eigenvalues of $X_{\gamma'}(T)$ and $(\eta_1(\gamma'),\ldots,\eta_N(\gamma'))$ the respective eigenspaces of $M(\gamma')$, such that the $\mathbb{C}$-span of the vector spaces $\eta_1(\gamma'),\ldots,\eta_N(\gamma')$ is equal to $\Span_\mathbb{C}(\eta_1(\gamma'),\ldots,\eta_N(\gamma')) = \mathbb{C}^N$.
    \end{definition}

    The following lemma will be at the heart of the existence of a continuously differentiably parametrized Floquet normal form.

    \begin{lemma}\label{lem:cont_diff_eig_val_vect}
        Let $A: \mathbb{R}^d \times \mathbb{R} \rightarrow \Mat_{N\times N}(\mathbb{C})$ be continuously differentiable function such that
            \begin{align*}
                A(\gamma,\cdot) \text{ is } T\text{-periodic for all } \gamma \in \mathbb{R}^d.
            \end{align*}
        Let $X_\gamma(t)$ be the fundamental solution associated to the parameterized lODE
            \begin{align*}
                \left\{\frac{d}{dt}X = A(\gamma,t)\right\}_{\gamma \in \mathbb{R}^d},
            \end{align*}
        and assume that the monodromy matrix
            \begin{align*}
                X(\cdot,T):\quad     &\mathbb{R}^d   &\longrightarrow    &\hspace{1cm} \GL_N(\mathbb{C}),\\
                           \quad     &\gamma         &\longmapsto        &\hspace{1cm} X(\gamma,t)
            \end{align*}
        is locally $C^1$-diagonalizable for all $\gamma \in \mathbb{R}^d$. Denote by 
        \begin{align*}
            \Lambda: \; \mathbb{R}^d \longrightarrow  \mathbb{C}^N
        \end{align*}
        a map that associates to each $\gamma \in \mathbb{R}^d$ the eigenvalues $\Lambda(\gamma)=(\Lambda_1(\gamma), \ldots, \Lambda_N(\gamma))$ of $X(\gamma,T)$, which are the characteristic multipliers of the parameterized lODE. Furthermore, denote by 
        \begin{align*}
            \eta: \; \mathbb{R}^d  \longrightarrow  (\mathbb{CP}^{N-1})^N
        \end{align*}
        the map that associates to each $\gamma \in \mathbb{R}^d$ the eigenspaces $\eta(\gamma)=(\eta_1(\gamma), \ldots, \eta_N(\gamma))$ of $X(\gamma,T)$, such that $\eta_n(\gamma)$ is the eigenspace corresponding to the eigenvalue $\Lambda_n(\gamma)$ and $\mathbb{CP}^{N-1}$ denotes $(N-1)$-dimensional complex projective space.

        Then the map $\Lambda\times\eta$ can be chosen to be continuously differentiable.
    \end{lemma}

    \begin{proof}
        This is a direct consequence of the $C^1$-diagonalizability of $X(\cdot,t)$ and the simply connectedness of $\mathbb{R}^d$.
    \end{proof}

    \begin{theorem}[$C^1$-parameterized Floquet normal form]\label{thm:C1-param_FNF_for_Rd-param_lODE}
        Let $A: \mathbb{R}^d \times \mathbb{R} \rightarrow \Mat_{N\times N}(\mathbb{C})$ be continuously differentiable function such that
            \begin{align*}
                A(\gamma,\cdot) \text{ is } T\text{-periodic for all } \gamma \in \mathbb{R}^d.
            \end{align*}
        Let $X_\gamma(t)$ be the fundamental solution associated to the parameterized lODE
            \begin{align}\label{eq:Thm_Rd_param_lODE}
                \left\{\frac{d}{dt}X = A(\gamma,t)\right\}_{\gamma \in \mathbb{R}^d},
            \end{align}
        and assume that the monodromy matrix
            \begin{align*}
                X(\cdot,T): \quad     &\mathbb{R}^d   &\longrightarrow    &\hspace{1cm} \GL_N(\mathbb{C}),\\
                            \quad     &\gamma         &\longmapsto        &\hspace{1cm} X(\gamma,t)
            \end{align*}
        is locally $C^1$-diagonalizable for all $\gamma \in \mathbb{R}^d$.

        Then there exists a continuously differentiably parameterized Floquet normal form associated to the lODE \eqref{eq:Thm_Rd_param_lODE}. In fact, the following holds:
            \begin{enumerate}
                \item There exists a continuously differentiable parameterization of the characteristic multipliers $$(\lambda_1(\gamma),\ldots,\lambda_N(\gamma))$$ and corresponding eigenspaces $(\eta_1(\gamma),\ldots,\eta_N(\gamma))$ associated to the system \eqref{eq:Thm_Rd_param_lODE}.
                \item Fixing such a continuously differentiable parametrization, there exists a choice of Floquet exponents $(\mu_1(\gamma),\ldots,\mu_N(\gamma))$ which will satisfy $\exp(\mu_n T) = \lambda_n$ for $n \in \{1,\ldots,N\}$ and which continuously differentiably depends on $\gamma \in \mathbb{R}^d$.
                \item Fixing such a continuously differentiable parametrization, a continuously differentiably parameterized choice of Floquet normal form of the system \eqref{eq:Thm_Rd_param_lODE} is given by 
                            \begin{align*}
                                F(\gamma)~ ~ &= \begin{pmatrix}
                                                    | & &| \\
                                                    \eta_1 &\ldots &\eta_N \\
                                                    | & &| \\
                                                \end{pmatrix}\begin{pmatrix}
                                                    \mu_1 & & \\
                                                    &\ddots & \\
                                                    & &\mu_N \\
                                                \end{pmatrix}\begin{pmatrix}
                                                    | & &| \\
                                                    \eta_1 &\ldots &\eta_N \\
                                                    | & &| \\
                                                \end{pmatrix}^{-1},  \\
                                P(\gamma,t) &= X_\gamma(t)\exp(-tF).
                            \end{align*}
            \end{enumerate}
    \end{theorem}

    \begin{proof}
        The statement of the theorem mainly follows from the construction of a Floquet normal form in the beginning of this subsection together with  Lemma \ref{lem:cont_diff_eig_val_vect}.
        
        The first part is precisely the statement of Lemma \ref{lem:cont_diff_eig_val_vect}, which gives a continuously differentiable parameterization of the characteristic multipliers and associated eigenspaces.

        In part two, the only thing that remains to be justified is the existence of a continuously differentiable choice of Floquet exponents associated to the given choice of characteristic multipliers $(\lambda_1(\gamma),\ldots,\lambda_N(\gamma))$. Its existence is due to the fact that the exponential map gives a universal cover of the space $\mathbb{C}^*$ and due to the simply connectedness of $\mathbb{R}^d$, it is possible to lift the maps $\lambda_n: \mathbb{R}^d \rightarrow \mathbb{C}^*$ to maps $\mu_n: \mathbb{R}^d \rightarrow \mathbb{C}$ via the exponential map $\exp(T(\cdot))$. In other words, for any $n \in \{1,\ldots,N\}$ there exists a continuously differentiable map $\mu_n$ which makes the following diagram commuting:
                \begin{align*}
                    \begin{xy}
                        \xymatrix{
                                                                                    &&& \mathbb{C} \ar[dd]^{\exp(T(\cdot))}    \\
                            \mathbb{R}^{d} \ar[urrr]^{\mu_n} \ar[drrr]_{\lambda_n}                        \\
                                                                                    &&& \mathbb{C}^*,
                        }
                    \end{xy}
                \end{align*}

        Part three follows when the procedure, which was explained in the first half of this subsection, is applied to the continuously differentiable parameterization of a choice of Floquet exponents of part two and to the continuously differentiable parameterization of a choice of eigenspaces of part one.
    \end{proof}


    \subsection{Topological phenomena of parameterized Floquet normal forms}\label{subsec:Derivation_of_topological_invariants}

        In this subsection we will come back to the setting where the considered parameterized periodic lODE 
            \begin{align}\label{eq:Td_param_lODE}
                \left\{\frac{d}{dt}x = A(\alpha,t)x \right\}_{\alpha \in \mathbb{T}^d}
            \end{align}
        is periodically parameterized by a torus $\mathbb{T}^d$, which is understood as $\mathbb{T}^d = \mathbb{R}^d/L$ for some fixed lattice $L \subset \mathbb{R}$. 

        The parameterized system \eqref{eq:Td_param_lODE} can be rewritten as 
            \begin{align}\label{eq:Rd_param_lODE}
                \left\{\frac{d}{dt}x = A(\gamma,t)x \right\}_{\gamma \in \mathbb{R}^d}
            \end{align}
        together with the identification $ A(\gamma,t) = A(\alpha,t) $ for $\gamma = \alpha \mod L$ and $\alpha \in \mathbb{T}^d$.

        Assume that the monodromy matrix associated to the system \eqref{eq:Rd_param_lODE} is locally $C^1$-diagonalizable. Then it follows from Theorem \ref{thm:C1-param_FNF_for_Rd-param_lODE} that there exist continuously differentiable parameterizations of the characteristic multipliers
            \begin{align*}
                (\lambda_n: \mathbb{R}^d \rightarrow \mathbb{C}^*)_{n = 1}^N
            \end{align*}
        with corresponding eigenspaces
            \begin{align*}
                (\eta_n: \mathbb{R}^d \rightarrow \mathbb{CP}^{N - 1})_{n = 1}^N.
            \end{align*}
        Being eigenvalues and eigenvectors of the $L$-periodic map $X(T,\cdot)$, they fulfill the following property:
            \begin{align*}
                \gamma \in \mathbb{R}^d     &\qquad \longmapsto   &\{\lambda_1(\gamma), \ldots, \lambda_N(\gamma)\}  &\in \{A \subset \mathbb{R}^N\} \\
                \gamma \in \mathbb{R}^d     &\qquad \longmapsto   &\{\eta_1(\gamma), \ldots, \eta_N(\gamma)\}        &\in \{A \subset \mathbb{CP}^{N - 1}\}
            \end{align*}
        are also $L$-periodic. In other words, the set of eigenvalues $\{\lambda_1(\gamma), \ldots, \lambda_N(\gamma)\}$ and the set of eigenspaces $\{\eta_1(\gamma), \ldots, \eta_N(\gamma)\}$ are $L$-periodic. However, the maps $\lambda_n: \mathbb{R}^d \rightarrow \mathbb{C}^*$ do not need to be $L$-periodic. Since there are only finitely many eigenvalues and since the set of eigenvalues is $L$-periodic, it follows that for every $n \in \{1,\ldots,N\}$ there exists a lattice $L_n \subset L$ such that $\lambda_n$ is $L_n$-periodic.
        Choosing this lattice to be a maximal sublattice $L_n \subset L$ such that $\lambda_n$ is $L_n$-periodic, makes the lattice $L_n$ unique and thus allows to associated the following well-defined topological invariant to the eigenvalue $\lambda_n: \mathbb{R}^d \rightarrow \mathbb{C}^*$.

        \begin{definition}
            Let 
                \begin{align*}
                    \left\{\frac{d}{dt}x = A(\alpha,t)x \right\}_{\alpha \in \mathbb{T}^d}
                \end{align*}
            be a periodically parameterized periodic lODE and let $T$ be the period of $A(\gamma,\cdot)$. If the monodromy matrix $X_\alpha(T)$ is locally $C^1$-diagonalizable and $(\lambda_1(\gamma), \ldots, \lambda_N(\gamma))$ is a continuously differentiable parameterization of the characteristic multipliers, then the \emph{Type I.a topological invariant} associated to $\lambda_n$ is defined as the quotient
                $$ L/L_n, $$
            where $L \subset \mathbb{R}$ is the lattice associated to the parameter space $ \mathbb{T}^d$  of the parameterized lODE and $L_n \subset L$ is the maximal sublattice of $L$ such that $\lambda_n$ is $L_n$-periodic.
        \end{definition}

        This topological effect indicates whether the different bands corresponding to the characteristic multipliers 'brade'. We call \emph{brading} the effect shown in Figure \ref{fig:brading} for the case where the parameter space is one-dimensional. Brading describes the effect when different bands $\lambda_n$ and $\lambda_{n'}$ are not $L$-periodic but only the set $\{\lambda_n(\gamma),\lambda_{n'}(\gamma)\}$ is. In that case, we say that $\lambda_n$ and $\lambda_{n'}$ brade. This effect can involve multiple bands, see e.g. subsection  \ref{subsec:TypeI.a_top_eff}.

        \begin{figure}[h]
            \caption{Non-trivial Type I.a topological invariant, also called brading.}\label{fig:brading}
            \centering
            \includegraphics{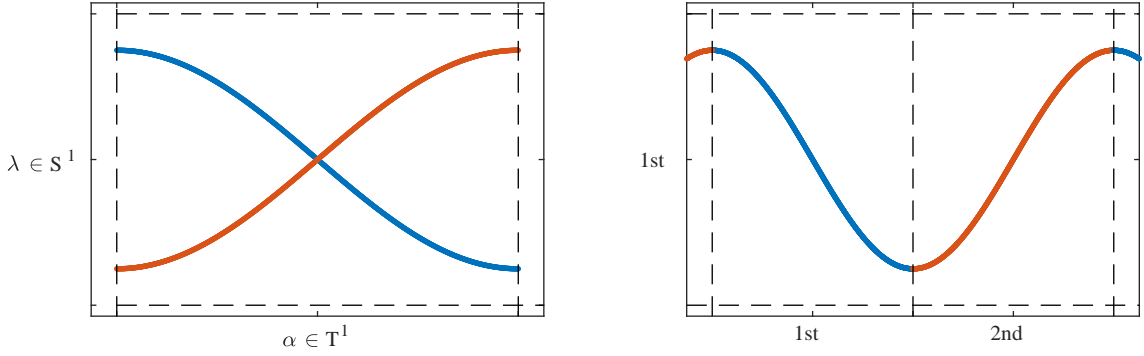}
        \end{figure}

        For a given band $\lambda_n$ with corresponding maximal sublattice $L_n \subset L$, it is possible to interpret it as well as its associated eigenspace $\eta_n$ as functions
            \begin{align*}
                \lambda_n :     &\quad\mathbb{R}^d/L_n     &\longrightarrow    &\quad\quad\mathbb{C}^*,      \\
                \eta_n :        &\quad\mathbb{R}^d/L_n     &\longrightarrow    &\quad\quad\mathbb{CP}^{N-1}, \\
            \end{align*}
        on the torus $\mathbb{T}^d_n := \mathbb{R}^d/L_n$ and to consider the associated homotopy classes of those functions. This leads to the following two definitions.

        \begin{definition}[Type I.b homotopy class]\label{def:TypeIb}
            Let 
                \begin{align*}
                    \left\{\frac{d}{dt}x = A(\alpha,t)x \right\}_{\alpha \in \mathbb{T}^d}
                \end{align*}
            be a periodically parameterized periodic lODE and let $T$ be the period of $A(\gamma,\cdot)$. If the monodromy matrix $X_\alpha(T)$ is locally $C^1$-diagonalizable and $(\lambda_1(\gamma), \ldots, \lambda_N(\gamma))$ is a continuously differentiable parameterization of the characteristic multipliers and $(\eta_1(\gamma), \ldots, \eta_N(\gamma))$ are the corresponding parameterized eigenspaces, then the \emph{Type I.b homotopy class} associated to $\eta_n$ is defined as the homotopy class
            \begin{align*}
                \eta_n \in  [\mathbb{T}^d_n, \mathbb{CP}^{N-1}],
            \end{align*}
            where $\mathbb{T}^d_n$ is the torus associated to $\lambda_n$, that is, $\mathbb{T}^d_n$ is defined as $\mathbb{R}^d/L_n$ where $L_n \subset L$ is the maximal sublattice such that $\lambda_n$ is $L_n$-periodic.
        \end{definition}

        \begin{definition}[Type II.a homotopy class]
            Let 
                \begin{align*}
                    \left\{\frac{d}{dt}x = A(\alpha,t)x \right\}_{\alpha \in \mathbb{T}^d}
                \end{align*}
            be a periodically parameterized periodic lODE and let $T$ be the period of $A(\gamma,\cdot)$. If the monodromy matrix $X_\alpha(T)$ is locally $C^1$-diagonalizable and $(\lambda_1(\gamma), \ldots, \lambda_N(\gamma))$ is a continuously differentiable parameterization of the characteristic multipliers, then the \emph{Type II.a homotopy class} associated to $\lambda_n$ is defined as the homotopy class
            \begin{align*}
                \lambda_n \in  [\mathbb{T}^d_n, \mathbb{C}^*],
            \end{align*}
            where $\mathbb{T}^d_n$ is the torus associated to $\lambda_n$, that is, $\mathbb{T}^d_n$ is defined as $\mathbb{R}^d/L_n$ where $L_n \subset L$ is the maximal sublattice such that $\lambda_n$ is $L_n$-periodic.
        \end{definition}

        The Type I and Type II terminology is due to the nature of the topological effects. Topological effects I.a and I.b can already occur in the case of a time-independent coefficient matrix $A(t,\gamma) = A_{const}(\gamma)$, whereas Type II topological invariants are unique to the setting when the coefficient matrix $A(t,\gamma)$ is time-dependent. 
        
        The topological invariants which have been defined so far are all linked to the Floquet exponent matrix $F$ or actually the monodromy matrix $X(T,\cdot)$ and are independent of the Lyapunov transformation $P(t,\gamma)$. However, the Type II.a homotopy class will indicate whether $P$ can be defined periodically with respect to $\gamma$.

        The reason why a non-trivial Type II.a invariant implies non-periodic Lyapunov transformation becomes clear already when one considers a one-dimensional parameter set $\mathbb{T}^1$ of the parameterized periodic lODE $\{\frac{d}{dt}x = A(\alpha,t)x\}_{\alpha \in \mathbb{T}^1}$. Assume $\lambda_n(\gamma)$ is a characteristic multiplier of the parameterized lODE and is $L_n$-periodic with $L_n \subset L$ and has non-trivial homotopy class $\lambda_n \in [\mathbb{R}/L_n, \mathbb{C}^*]$. This means that the image of the circle $\mathbb{R}/L_n$ winds around the origin in $\mathbb{C}^*$ and that is why we also say that the band associated to a characteristic multiplier \emph{winds} or has \emph{non-trivial winding} whenever that characteristic multiplier has non-trivial Type II.a.

        \begin{figure}[h]
            \caption{Non-trivial Type II.a topological invariant, also called winding.}\label{fig:winding}
            \centering
            \includegraphics{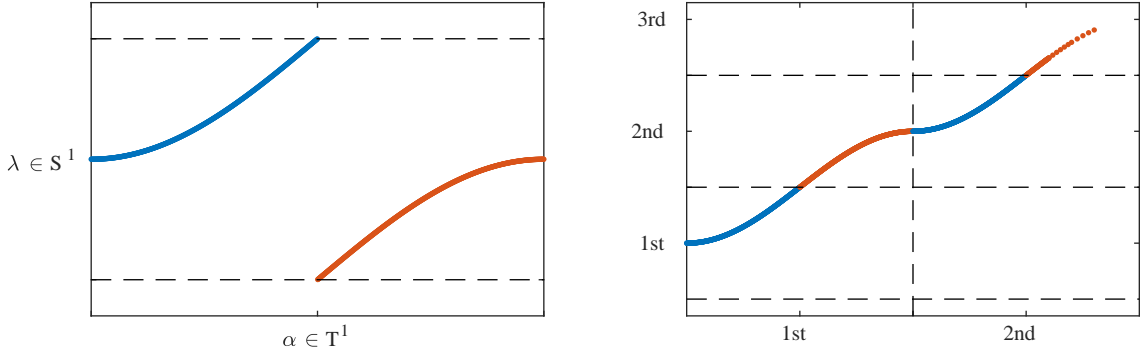}
        \end{figure}

        Winding has the following effect on the Floquet exponent matrix or more precisely on the characteristic exponents $\mu_n$. The characteristic exponents are obtained via lifting the characteristic multipliers in the following way 
            \begin{align*}
                \begin{xy}
                    \xymatrix{
                                                                                &&& \mathbb{C} \ar[dd]^{\exp(T(\cdot))}    \\
                        \mathbb{R}^{d} \ar[urrr]^{\mu_n} \ar[drrr]_{\lambda_n}                        \\
                                                                                &&& \mathbb{C}^*.
                    }
                \end{xy}
            \end{align*}
        A non-trivial winding of $\lambda_n$ will then lead to a non-periodic lift $\mu_n$. This phenomenon is depicted in Figure \ref{fig:winding}, namely, $\mu_n$ gains a non-zero phase for every fundamental domain of $L_n$ that it traverses, leading to a linearly increasing characteristic exponent. This then implies that also the Floquet exponent matrix $F$ will not be periodic but will have growing/decaying eigenvalues. Finally this property propagates to the Lyapunov transformation $P$. In the case of non-trivial winding also $P$ will not be periodic with respect to $\gamma \in \mathbb{R}$.

        Since, to our knowledge, the phenomenon of non-trivial winding in the case of subwavelength solutions for Floquet metamaterials hasn't been observed yet, we will restrict this work to the case where the Type II.a homotopy class is trivial. A sufficient condition for this assumption to be valid would be that the coefficient matrix $A(\gamma,t)$ commutes with its derivative $\frac{d}{dt}A(\gamma,t)$ for all $(\gamma,t) \in \mathbb{R}^{d+1}$.

        \begin{lemma}[Periodic Lyapunov transform]\label{lem:periodic_Lyapunov_transform}
            Let $A: \mathbb{R}^d \times \mathbb{R} \rightarrow \Mat_{N\times N}(\mathbb{C})$ be continuously differentiable function such that
                \begin{align*}
                    A(\gamma,\cdot) \text{ is } T\text{-periodic for all } \gamma \in \mathbb{R}^d \text{ and such that }
                    A(\cdot,t) \text{ is } L\text{-periodic for all } t \in \mathbb{R},
                \end{align*}
            where $L \subset \mathbb{R}$ is a lattice.
            Let $X_\gamma(t)$ be the fundamental solution associated to the parameterized lODE
                \begin{align*}
                    \left\{\frac{d}{dt}X = A(\gamma,t)\right\}_{\gamma \in \mathbb{R}^d},
                \end{align*}
            and assume that the monodromy matrix
                \begin{align*}
                    X(\cdot,T): \quad     &\mathbb{R}^d   &\longrightarrow    &\hspace{1cm} \GL_N(\mathbb{C}),\\
                                \quad     &\gamma         &\longmapsto        &\hspace{1cm} X(\gamma,t)
                \end{align*}
            is locally $C^1$-diagonalizable for all $\gamma \in \mathbb{R}^d$. Let $(\lambda_1(\gamma),\ldots,\lambda_N(\gamma))$ denote a continuously differentiable parameterization of the characteristic multipliers with corresponding eigenspaces $(\eta_1(\gamma),\ldots,\eta_N(\gamma))$ and maximal lattices $(L_1,\ldots,L_N)$. Let $(\mu_1(\gamma),\ldots,\mu_N(\gamma))$ be a continuously differentiable choice of characteristic exponents and assume that for all $n \in \{1,\ldots,N\}$ the Type II.a of $\lambda_n$ is trivial. Then the associated Floquet exponent matrix $F(\gamma)$ and the corresponding Lyapunov transform $P(\gamma,t)$ are $\tilde{L}$-periodic, where
                $$\tilde{L} = \bigcap_{n=1}^N L_n.$$
        \end{lemma}

        \begin{proof} It needs to be proven that the $(\mu_1(\gamma),\ldots,\mu_N(\gamma))$ defines a $\tilde{L}$-periodic map. Indeed, this being proven it follows from 
            \begin{align*}
                F(\gamma)~ ~ &= \begin{pmatrix}
                                    | & &| \\
                                    \eta_1 &\ldots &\eta_N \\
                                    | & &| \\
                                \end{pmatrix}\begin{pmatrix}
                                    \mu_1 & & \\
                                    &\ddots & \\
                                    & &\mu_N \\
                                \end{pmatrix}\begin{pmatrix}
                                    | & &| \\
                                    \eta_1 &\ldots &\eta_N \\
                                    | & &| \\
                                \end{pmatrix}^{-1} \text{ and }  \\
                P(\gamma,t) &= X_\gamma(t)\exp(-tF)
            \end{align*}
            that also $F$ and $P$ are $\tilde{L}$-periodic.
            
            Let $n \in \{1,\ldots,N\}$. Then $\mu_n$ is characterized by the equation $\exp(\mu_n) = \lambda_n$ and its value $\mu_n(1)$. Interpreting $\exp: \mathbb{C} \rightarrow \mathbb{C}^*$ as a universal cover, $\mu_n: \mathbb{R}^d \rightarrow \mathbb{C}$ is then given as a lifting of $\lambda_n: \mathbb{R}^d \rightarrow \mathbb{C}^*$. Since $\lambda_n$ is $L_n$-periodic and contractible, it follows that also $\mu_n$ is $L_n$-periodic. Thus proving that $(\mu_1(\gamma),\ldots,\mu_N(\gamma))$ is $\tilde{L}$-periodic.
        \end{proof}

        \begin{definition}[Type II.b homotopy class]\label{def:TypeIIb}
            Let $A: \mathbb{R}^d \times \mathbb{R} \rightarrow \Mat_{N\times N}(\mathbb{C})$ be continuously differentiable function such that
            \begin{align*}
                A(\gamma,\cdot) \text{ is } T\text{-periodic for all } \gamma \in \mathbb{R}^d \text{ and such that }
                A(\cdot,t) \text{ is } L\text{-periodic for all } t \in \mathbb{R},
            \end{align*}
        where $L \subset \mathbb{R}$ is a lattice.
        Let $X_\gamma(t)$ be the fundamental solution associated to the parameterized lODE
            \begin{align*}
                \left\{\frac{d}{dt}X = A(\gamma,t)\right\}_{\gamma \in \mathbb{R}^d},
            \end{align*}
        and assume that the monodromy matrix
            \begin{align*}
                X(\cdot,T): \quad     &\mathbb{R}^d   &\longrightarrow    &\hspace{1cm} \GL_N(\mathbb{C}),\\
                            \quad     &\gamma         &\longmapsto        &\hspace{1cm} X(\gamma,t)
            \end{align*}
        is locally $C^1$-diagonalizable for all $\gamma \in \mathbb{R}^d$. Let $(\lambda_1(\gamma),\ldots,\lambda_N(\gamma))$ denote a continuously differentiable parameterization of the characteristic multipliers with corresponding eigenspaces $(\eta_1(\gamma),\ldots,\eta_N(\gamma))$ and maximal lattices $(L_1,\ldots,L_N)$. Let $(\mu_1(\gamma),\ldots,\mu_N(\gamma))$ be a continuously differentiable choice of characteristic exponents and assume that for all $n \in \{1,\ldots,N\}$ the Type II.a of $\lambda_n$ is trivial, then the \emph{Type II.b homotopy class} associated to the corresponding Lyapunov transform $P$ is defined as the homotopy class 
                $$ P \in [\mathbb{R}^d/\tilde{L} \times \mathbb{R}/T\mathbb{Z}, \GL_N(\mathbb{C})],$$
                where $\tilde{L} := \cap_{1=n}^N L_n$ is a lattice in $\mathbb{R}^d$.
        \end{definition}

        Using polar decomposition of matrices and the fact that the space of positive definite matrices is contractible, it is possible to rephrase the definition of the Type II.b homotopy class. Indeed, it suffices to examine the homotopy class of the unitary part $U(\gamma,t)$ of $P(\gamma,t)$ in 
            $$ [\mathbb{R}^d/L \times \mathbb{R}/T\mathbb{Z}, \U(N)],$$
        where $\U(N)$ denotes the space of $N\times N$-dimensional unitary matrices.
        Denoting by $\SU(N)$ the space of special unitary $N \times N$-dimensional matrices and using the diffeomorphism
            \begin{align*}
                \U(N) \quad\quad\quad   &\longrightarrow        &&\qquad\qquad\qquad\quad \mathbb{S}^1 \times \SU(N) \\
                V \quad\quad\quad       &\longmapsto            &&\left(\det(V), \begin{pmatrix} 
                                                                    | & | &  &| \\
                                                                    \frac{1}{\det(V)}V^{(1)} &V^{(2)} & \ldots &V^{(N)} \\
                                                                    | & | &  &| 
                                                                \end{pmatrix} \right),
            \end{align*}
        the homotopy class of $P \in [\mathbb{R}^d/L \times \mathbb{R}/T\mathbb{Z}, \GL_N(\mathbb{C})]$ is uniquely determined by the associated homotopy classes in $[\mathbb{R}^d/L \times \mathbb{R}/T\mathbb{Z}, \mathbb{S}^1]$ and $[\mathbb{R}^d/L \times \mathbb{R}/T\mathbb{Z}, \SU(N)]$. More precisely, the homotopy class of $P$ is uniquely determined by the functions
            \begin{align*}
                \frac{\det(P(\gamma,t)}{\abs{\det(P(\gamma,t)}} \in C^0(\mathbb{R}^d/L \times \mathbb{R}/T\mathbb{Z}, \mathbb{S}^1) 
            \end{align*}
        and
            \begin{multline}\label{eq:def_of_SU_part}
                S(\gamma,t) :=\\ \begin{pmatrix} 
                    | & | &  &| \\
                    \frac{\lvert \det(P) \rvert}{\det(P)}(P (P^*P)^{-\frac{1}{2}})^{(1)} &(P (P^*P)^{-\frac{1}{2}})^{(2)} & \ldots &(P (P^*P)^{-\frac{1}{2}})^{(N)} \\
                    | & | &  &| 
                \end{pmatrix} \in C^0(\mathbb{R}^d/L \times \mathbb{R}/T\mathbb{Z}, \SU(N)).
            \end{multline}
        
        The following lemma will further simplify the homotopy class of the $\det(P)/\abs{\det(P)}$. Namely, it will be proven that the homotopy class of $\det(P)/\abs{\det(P)}$ solely depends on $\det(P(\gamma,\cdot))/\abs{\det(P(\gamma,\cdot))} \in C^0(\mathbb{R}/T\mathbb{Z}, \mathbb{S}^1)$ for any choice of fixed $\gamma \in \mathbb{R}^d/\tilde{L}$.
        
        \begin{lemma}\label{lem:homotopy_class_of_Delta_o_P}
            Let $A:\mathbb{R}^d \times \mathbb{R} \rightarrow \Mat_{N\times N}(\mathbb{C})$ be a $L \times T\mathbb{Z}$-periodic, continuously differentiable function. Let $P: \mathbb{R}^d/L \times \mathbb{R}/T\mathbb{Z} \rightarrow \GL_N(\mathbb{C})$ be the Lyapunov transform obtained in Lemma \ref{lem:periodic_Lyapunov_transform} and define
            \begin{align*}
                \Delta:\quad\quad       &\GL_N(\mathbb{C})  &\longrightarrow    &\qquad\qquad\qquad\qquad\mathbb{S}^1\\
                                        &M                  &\longmapsto        &\qquad\qquad\qquad\qquad\frac{\det(M)}{\abs{\det(M)}},\\
                \\
                \iota:\quad\quad        &\mathbb{S}^1       &\longrightarrow    &\qquad\qquad\qquad\qquad\mathbb{T}^{d+1}\\
                                        &\overline{t}       &\longmapsto        &\qquad\qquad\qquad\qquad(0,\ldots,0,\overline{t}).
            \end{align*}
            Then the homotopy class of $\Delta \circ P = \frac{\det(P)}{\abs{\det(P)}} \in C^0(\mathbb{R}^d/L \times \mathbb{R}/T\mathbb{Z}, \mathbb{S}^1)$ is uniquely determined by $\Delta \circ P \circ \iota$ in $\pi^1(\mathbb{S}^1)$.           
        \end{lemma}

        \begin{proof}
            In fact, the homotopy class of a continuous map $f: \mathbb{T}^{d+1} \rightarrow \mathbb{S}^1$ is uniquely determined by the homotopy classes of $f\circ \iota_j : \mathbb{S}^{1} \rightarrow \mathbb{S}^1$, where $\iota_j: \mathbb{S}^1 \hookrightarrow \mathbb{T}^{d+1}\cong (\mathbb{S}^1)^{d+1}$ is the inclusion in the $j$th factor. Thus  the homotopy class of $\Delta \circ P$ is also uniquely determined by the homotopy classes of the maps $\Delta \circ P \circ \iota_1, \ldots, \Delta \circ P \circ \iota_{d+1} \in C^0(\mathbb{S}^1,\mathbb{S}^1)$. Since $P(\alpha,0) = \Id_N$ for all $\alpha \in \mathbb{T}^d$ it follows that the maps $\Delta \circ P \circ \iota_1, \ldots, \Delta \circ P \circ \iota_{d}$ are homotopically trivial and thus the homotopy class of $\Delta \circ P$ is uniquely determined by $\Delta \circ P \circ \iota_{d+1} = \Delta \circ P \circ \iota$.
        \end{proof}

        \begin{remark}\label{rem:theory_real_coef_mat}
            This argument (i.e. Lemma \ref{lem:homotopy_class_of_Delta_o_P} and the preceeding discussion) also holds in the setting where $F$ can chosen to be real-valued. In that case, $P(\alpha,t)$ takes values in $\GL_N(\mathbb{R})$ and one can replace $\U(N)$ and $\SU(N)$ by $\Orth(N)$ and $\SO(N)$ (the spaces of orthogonal and special orthogonal $N \times N$-dimensional matrices), respectively.
        \end{remark}


\section{Application to high-contrast hexagonal  structures}\label{sec:Application_to_high-contrast_acoustic_hexagonal_metamaterial_structures}
    In this section the derived topological invariants will be applied to the setting of a hexagonal structure, like it is present in graphene.
    Highly symmetric structures as honeycomb or square lattice structures allow for a reduction of the two dimensional Brillouin zone to a one dimensional symmetry curve (see Figure \ref{fig:brioullin_zone}). This then reduces the parameterization space from a two dimensional torus $\mathbb{T}^2$ to a one-dimensional circle $\mathbb{S}^1$, which in return limits the variety of homotopic effects, as will be displayed in subsection \ref{subsec:S1_param_top_invar}.
    
    In the following section we will present the setting for which we will apply the derived topological invariants. In the subsequent section we will briefly explain the computational procedure we use to compute the Floquet-Lyapunov decomposition, the Floquet exponents and modes.  Thereafter, examples of a time-modulated honeycomb structures which have nontrivial Type I.a Topological Invariants will be presented.

    When the setting will be presented it will become apparent that we are in the case of a second order lODE which is parameterized by a one-dimensional parameter space $\mathbb{S}^1$. That is why the defined topological invariants are analysed closer in subsection \ref{subsec:S1_param_top_invar}.
    

        \subsection{Setting}\label{subsec:metamaterial_setting}
            A metamaterial is a prototype material of the following form. It is composed of two submaterials: the \emph{background} and the \emph{resonator material}. While the background material fills almost the whole space $\mathbb{R}^d$, in this case we will restrict us to $2$-dimensional space $\mathbb{R}^2$, the resonator material only occupies disconnected domains $D_1, \ldots, D_N  \subset \mathbb{R}^2$, which are repeated periodically with respect to some $\mathbb{R}$-linearly independent lattice vectors $g_1, g_2 \in \mathbb{R}^2$, which generate the lattice $G = g_1\mathbb{Z} \oplus g_2\mathbb{Z}$, such that the domain of the resonator material is given by
            $$\mathcal{D} =\bigcup_{g \in G}(g + D_1 \cup \ldots \cup D_N),$$
            see e.g. Figure \ref{fig:hexagonal_structure_metamaterial} for the case of a hexagonal structure material.
            The \emph{reciprocal lattice} is then defined as
                \begin{align*}
                    L := \left\{l \in \mathbb{R}^2 | <l,g>\, \in 2\pi\mathbb{Z} \text{ for all } g \in G \right\},
                \end{align*}
            where $<l,g>$ denotes the standard scalar product in $\mathbb{R}^2$.
            Background and resonator material are characterized by their corresponding material parameters $\rho$ and $\kappa$ which correspond to the density and the bulk modulus in the setting of scalar acoustic waves. To be precise, the density $\rho$ and the bulk modulus $\kappa$ are defined as 
                \begin{align} \label{eq:resonatormod}
                    \kappa(x,t) &=   \begin{cases}
                                        \kappa_0,           & x \in \R^2 \setminus \overline{\mathcal{D}}, \\
                                        \kappa_n(t),        & x\in g + D_n,\text{ with }  n \in \{1,\ldots,N\}, \, g \in G
                                    \end{cases} \\
                    \rho(x,t)   &=   \begin{cases}
                                        \rho_0,             & x \in \R^2 \setminus \overline{\mathcal{D}}, \\
                                        \rho_n(t),          & x\in g + D_n,\text{ with }  n \in \{1,\ldots,N\}, \, g \in G.
                                    \end{cases}
                \end{align}
            That is, the density $\rho$ and the bulk modulus $\kappa$ are piecewise constant in space and also in time in the domain of the background material. However, the material parameters are time-dependent inside the resonators.
            The goal is to study subwavelength solutions to the  wave equation with time-dependent coefficients 
                \begin{equation}\label{eq:wave}
                    \left(\frac{\p }{\p t } \frac{1}{\kappa(x,t)} \frac{\p}{\p t} - \nabla_x \cdot \frac{1}{\rho(x,t)} \nabla_x\right) u(x,t) = 0, 
                    \quad x\in \R^2,\, t\in \R.
                \end{equation}
            To this end we will restrict ourselves to the study of its associated \emph{subwavelength quasifrequencies}, when $\kappa$ and $\rho$ are $T$-periodic in time, in which case we understand the following by subwavelength quasifrequencies.   

            When the wave equation \eqref{eq:wave} is periodic in time and space, one can apply the Floquet transform with respect to time and space to the wave equation \eqref{eq:wave}. This leads to a parameterized set of problems with restricted solution spaces. Indeed, one obtains
            \begin{equation} \label{eq:wave_transf}
                \begin{cases}\ \ds \left(\frac{\p }{\p t } \frac{1}{\kappa(x,t)} \frac{\p}{\p t} - \nabla_x \cdot \frac{1}{\rho(x,t)} \nabla_x\right) u(x,t) = 0,\\ 
               \nm
                    \	u(x,t)e^{-i \omega t} \text{ is $T$-periodic in $t$}, \\
                    \nm
                    \	u(x,t)e^{-i \alpha \cdot x} \text{ is $G$-periodic in $x$}, 
                \end{cases}
            \end{equation}
            where $\omega$ ranges over the elements of the \emph{time-Brillouin zone} $Y_t^* := \mathbb{C} / (\Omega \Z)$ with $\Omega$ being the frequency of the material parameters, which is thus given by $\Omega = {2\pi}/{T}$, and $\alpha$ ranges over elements of the \emph{Brillouin zone} $\mathbb{R}^2/L$. If a solution $u(x,t)$ to \eqref{eq:wave_transf} exists for an $\omega \in Y_t^*$ and $\alpha \in \mathbb{R}^2/L$, then $u(x,t)$ is called a \emph{Bloch solution} and $\omega$ its associated \emph{(time-)quasifrequency} and $\alpha$ its associated \emph{(spatial) quasifrequency}.

            In order to study \emph{subwavelength} (time-)quasifrequencies, one needs to assume that the \emph{contrast parameter} 
            $$\delta := \frac{\rho_i(0)}{\rho_0}$$
            is small,\footnote{Supposing that $\rho_n(0) = \rho_{n'}(0)$ for all $n,n' \in \{1,\ldots,N\}$.} or, to be precise, one needs to consider solutions to the wave equation \eqref{eq:wave} as $\delta \rightarrow 0$.
            Assuming 
            \begin{equation} \label{eq:resonatormod_rho}
                \rho(x,t) = \begin{cases}
                                \rho_0,             & x \in \R^2 \setminus \overline{\mathcal{D}}, \\ 
                                \rho_r\rho_n(t),    & x\in g + D_n,\text{ with }  n \in \{1,\ldots,N\}, \, g \in G,
                            \end{cases}
            \end{equation}
            with $\rho_n(0) = 1$ for all $n = 1,\ldots, N$, one can regard the wave equation \eqref{eq:wave} as parameterized by the contrast parameter $\delta$ and one can consider its solutions as $\delta \rightarrow 0$. This is called the \emph{high contrast regime}. In the setting where the modulation frequency $\Omega$ may also depend on $\delta$, subwavelength frequencies are introduced as in \cite{ammari2020time} and are given by the following definition.

            \begin{definition}[Subwavelength quasifrequency] \label{def:sub}
                A quasifrequency $\omega = \omega(\delta) \in Y^*_t$ of \eqref{eq:wave_transf} is said to be a \emph{subwavelength quasifrequency} if there is a corresponding Bloch solution $u(x,t)$, depending continuously on $\delta$, which can be written as
                $$u(x,t)= e^{i \omega t}\sum_{m = -\infty}^\infty v_m(x)e^{i m\Omega t},$$
                where 
                $$\omega(\delta) \rightarrow 0 \in Y_t^* \text{ and }  M(\delta)\Omega(\delta) \rightarrow 0  \in \mathbb{R} \text{ as }  \delta \to 0,$$
                for some integer-valued function $M=M(\delta)$ such that, as $\delta \to 0$, we have
                $$\sum_{m = -\infty}^\infty \|v_m\|_{L^2(D)} = \sum_{m = -M}^M \|v_m\|_{L^2(D)} + o(1).$$
            \end{definition}

            In \cite{ammari2020time}, a capacitance matrix formulation to the subwavelength quasifrequencies as $\delta \rightarrow 0$ was proven. It describes subwavelength quasifrequencies as solutions to a (finite dimensional) linear system of equations in the setting where the density and bulk modulus are constant with respect to time. In the setting of time-dependent density and bulk modulus, the subwavelength quasifrequencies are described by a periodically parameterized lODE.
            
            In order to state this result, we need the following definitions of the \emph{time-dependent contrast parameters}, \emph{wave speed} and \emph{time-dependent wave speeds}
                \begin{align*}
                    \delta_n(t) = \frac{\rho_n(t)}{\rho_0}, 
                    \quad v_0 = \sqrt{\frac{\kappa_0}{\rho_0}}, 
                    \quad v_n(t) = \sqrt{\frac{\kappa_n(t)}{\rho_n(t)}},    
                \end{align*}
            respectively, with $n =1,\ldots,N$.

            \begin{theorem}\label{thm:EriksMasterEquation}
                Being in the high contrast regime of \eqref{eq:wave}, assume that the material parameter $\kappa$ is given by \eqref{eq:resonatormod}, that $\rho$ is given by \eqref{eq:resonatormod_rho} and that they satisfy 
                    $$\frac{1}{\rho_n(t)} = \sum_{m = -M}^M r_{n,m} e^{i m \Omega t}, \qquad \frac{1}{\kappa_n(t)} = \sum_{m = -M}^M k_{n,m} e^{i m \Omega t},$$
                for some $M(\delta)\in \N$ satisfying $M(\delta) = O\left(\delta^{-\gamma/2}\right)$ as $\delta \rightarrow 0$ for some fixed $0<\gamma<1$. Furthermore, suppose that the associated time-dependent contrast parameters, wave speed and time-dependent wave speeds satisfy for all $t\in \R$ and $n=1,\ldots,N$,
                    $$\delta_n(t) = O(\delta), \quad v = O(1), \quad v_n(t) = O(1) \quad \text{as } \hspace{0.2cm} \delta \rightarrow 0,$$
                respectively.
                Then, as $\delta \to 0$, the subwavelength quasifrequencies $\omega(\delta) \in Y^*_t$ to the wave equation \eqref{eq:wave} in the high contrast regime are, to leading order in $\delta$, given by the quasifrequencies of the system of ordinary differential equations in $y(t) = (y_n(t))_n$,
                    \begin{equation}\label{eq:C_lODE}
                        \sum_{m=1}^N C^\alpha_{nm} y_m(t) = -\frac{|D_n|}{\delta_n(t)}\frac{\dx}{\dx t}\left(\frac{1}{\delta_nv_n^2}\frac{\dx (\delta_ny_n)}{\dx t}\right),
                    \end{equation}
                for $n=1,\ldots,N$, where $C^\alpha =(C^\alpha_{nm})_{n,m}$ denotes the capacitance matrix associated to the infinite periodic structure of the considered metamaterial and the spatial quasifrequency $\alpha \in \mathbb{R}^2/L$.
            \end{theorem}

            The \emph{capacitance matrix} is a way to encode the geometry of  an infinite periodic structure into a square matrix. It has the same dimension as the total number of resonators in a fundamental domain. The capacitance matrix theory in the high contrast regime was developed in \cite{ammari2021functional}. It was first derived using \emph{Gohberg-Sigal theory} and \emph{layer potential techniques} for the finite structure case, where the resonators $D_1,\ldots, D_N$ are not repeated periodically in space, but where the resonator domain $\mathcal{D}\subset \mathbb{R}^3$ is given by $\mathcal{D} = D_1\cup\ldots\cup D_N$. Then, \emph{Floquet-Bloch theory} allowed to extend the results to  infinite structures. 

            One can rewrite \eqref{eq:C_lODE} into the following system of Hill equations:
            \begin{equation}\label{eq:hill}
                \Psi''(t)+ M(t)\Psi(t)=0,
            \end{equation}
            where the vector $\Psi$ and the matrix $M$ are defined as
                $$\Psi(t) = \left(\frac{\sqrt{\delta_n(t)}}{v_n(t)}y_n(t)\right)_{n=1}^N, \quad M(t) = W_1(t)C^\alpha W_2(t) + W_3(t)$$
            with $W_1, W_2$ and $W_3$ being diagonal matrices with corresponding diagonal entries
                $$\left(W_1\right)_{nn} = \frac{v_n\delta_n^{3/2}}{|D_n|}, \qquad \left(W_2\right)_{nn} =\frac{v_n}{\sqrt{\delta_n}}, \qquad \left(W_3\right)_{nn} = \frac{\sqrt{\delta_n}v_n}{2}\frac{\dx }{\dx t}\frac{1}{(\delta_nv_n^2)^{3/2}}\frac{\dx (\delta_nv_n^2)}{\dx t},$$
            for $n=1,\ldots,N$.


            \subsection{Computation of continuous parametrization of Floquet exponents and Floquet modes}\label{subsec:Computation_of_continuous_parametrization_of_Floquet_exponents_and_Floquet_modes}
            Since from a computational point of view, it is costly to compute matrix exponentials, one needs to develop another procedure to work around this problem. Since in the case of the Type I topological invariants, one needs to compute the Floquet exponents and Floquet modes, it is convenient to reuse those for the computation of $\exp(tF_\alpha)$ and, more importantly, for the computation of the Lyapunov transform $P(\alpha,t) := X_\alpha(t)\exp(-tF_\alpha)$. The details are given in the following lemma. 

            \begin{lemma}\label{lem:comp_for_cont_floq_lyap_decomp}
                Let 
                \begin{align*}
                    \left\{\frac{d}{dt}x = A(\alpha,t)x \right\}_{\alpha \in \mathbb{T}^d}
                \end{align*}
            be a periodically parameterized periodic lODE and let $T$ be the period of $A(\gamma,\cdot)$. If the monodromy matrix $X_\alpha(T)$ is locally $C^1$-diagonalizable. Let $\tilde{L}\subset L$ be a maximal lattice as in Lemma \ref{lem:periodic_Lyapunov_transform} and let 
                    \begin{align*}
                        \Lambda: \; \mathbb{R}^d/\tilde{L} \longrightarrow  \mathbb{C}^N
                    \end{align*}
                be a continuously differentiable map that associates to each $\gamma \in \mathbb{R}^d$ the characteristic multipliers $\Lambda(\gamma)=(\Lambda_1(\gamma), \ldots, \Lambda_N(\gamma))$. Furthermore, denote by 
                    \begin{align*}
                        \eta: \; \mathbb{R}^d  \longrightarrow  (\mathbb{CP}^{N-1})^N
                    \end{align*}
                a continuously differentiable map that associates to each $\gamma \in \mathbb{R}^d$ the eigenspaces $\eta(\gamma)=(\eta_1(\gamma), \ldots, \eta_N(\gamma))$, such that $\eta_n(\gamma)$ is the eigenspace corresponding to the eigenvalue $\Lambda_n(\gamma)$.
                Then the (continuously differentiable and periodic) Floquet-Lyapunov decomposition of $X_\alpha(t)$
                    \begin{align*}
                        X_\alpha(t) = P(\alpha,t)\exp\left(tF_\alpha\right),
                    \end{align*}
                as in Lemma \ref{lem:comp_for_cont_floq_lyap_decomp} is given by
                    \begin{align*}
                        F_\alpha &= \frac{1}{T}\int_0^TA(\alpha,t)dt,\\
                        P(\alpha,t) &= X_\alpha(t)\exp\left(-tF_\alpha\right)
                    \end{align*}
                and can be computed using the Floquet exponent matrix which satisfies 
                    \begin{align}\label{eq:F_eig_vect_val}
                        F_\alpha = V\diag\left(\frac{1}{T}\Lambda(\alpha)\right)V^{-1},
                    \end{align}
                with 
                    \begin{align*}
                        V = \begin{pmatrix}
                            |& &|\\
                            v_1 &\ldots &v_N\\
                            |& &|
                        \end{pmatrix},
                    \end{align*} 
                where $v_n$ is some generator of the vector space $V_n = \eta_n(\alpha)$ for $1 \leq n \leq N$. The Lyapunov transformation can then be computed as 
                    \begin{align}\label{eq:comp_P}
                        P(\alpha,t) = X_\alpha(t)V\diag\left(\exp\left(-\frac{1}{T}\Lambda(\alpha)\right)\right)V^{-1},
                    \end{align}
                where $\exp$ denotes the `usual' complex exponential function applied to each coordinate of $\Lambda(\alpha)$, and not the matrix analogue.
            \end{lemma}

            \begin{proof}
                Equation \eqref{eq:F_eig_vect_val} is simply a diagonalization of $F_\alpha$ and equation \eqref{eq:comp_P} is due to the fact that
                $$ \exp(-tF_\alpha) = \exp\left(V\diag\left(\frac{1}{T}\Lambda(\alpha)\right)V^{-1}\right) = V\diag\left(\exp\left(-\frac{1}{T}\Lambda(\alpha)\right)\right)V^{-1}.$$
            \end{proof}

            The above lemma is particularly useful when one wants to compute the Lyapunov transform numerically. In fact, the procedure we chose for this paper is the following. We computed the fundamental solution $X_\alpha(t)$ using an appropriate numerical method, then we estimated the eigenvalues $\xi_1(\alpha), \ldots, \xi_N(\alpha)$ with respective eigenvectors $v_1(\alpha), \ldots, v_N(\alpha)$ of $X_\alpha(T)$ and parameterized them in a continuously differentiable manner. Since, with the notation from subsection \ref{subsec:Continuously_differentiably_parameterized_Floquet_normal_form}, $\Xi(\beta) = (\xi_1(\beta), \ldots, \xi_N(\beta))$ relates to $\Lambda(\beta)$ in the following fashion
                \begin{align*}
                    \Xi(\beta) = \exp\left(\Lambda(\beta)\right), \text{ for } \beta \in \mathbb{R}^N/\tilde{L}
                \end{align*}
            we took the appropriate logarithm branch for each $\xi_n(\beta)$ in order to obtain
                \begin{align*}
                    V(\beta)\log(\Xi(\beta))V(\beta)^{-1} = \int_0^TA(\beta,t)dt.
                \end{align*}
            Thus choosing 
                \begin{align*}
                    \Lambda(\beta) = \log(\Xi(\beta)),
                \end{align*}
            one can compute $F(\alpha)$ as 
                \begin{align*}
                    F_\alpha = V(\beta)\frac{1}{T}\Lambda(\beta) V(\beta)^{-1}, \text{ where } \alpha = \zeta(\beta)
                \end{align*}
            and the Lyapunov transform as
                \begin{align*}
                    P(\alpha,t) = X_\alpha(t)V(\beta)\exp\left(-\frac{t}{T}\Lambda(\beta)\right)V(\beta)^{-1},
                \end{align*}
            where it doesn't matter which $\beta \in \mathbb{R}^N/\tilde{L}$ is chosen, as long as $\zeta(\beta) = \alpha$ is satisfied.


            \subsection{Analysis of topological effects in the setting of high-constrast metamaterials}\label{subsec:S1_param_top_invar}

            This section is dedicated to the analysis of the case where the lODE $\frac{d}{dt}X = A_\alpha(t)X$ is parameterized by $\alpha \in \mathbb{S}^1$, that is, where the parameter space $\mathbb{T}^d$ is of dimension $d = 1$.
            
            We will see that a one-dimensional parameter space considerably restricts the topological properties of the Floquet normal form of the lODE 
            \begin{align*}
                \left\{\frac{d}{dt}X = A_\alpha(t) X\right\}_{\alpha \in \mathbb{S}^1}.
            \end{align*}
            It will be proven that in that case only the Type I.a Topological Invariant remains of interest and that the Type II.b topological invariant is uniquely determined by the function $\det(P)/\lvert \det(P) \rvert$. In other words, the Type I.b and the homotopy class of $S(\gamma,t)$ are automatically trivial whenever the parameter space is equal to $\mathbb{S}^1$.
            Both results are due to the homotopical properties of $\mathbb{CP}^{N-1}$ and $\SU(N)$, respectively. 

            Considering the Type I.b Homotopy Class first, the following result holds.

            \begin{lemma}\label{lem:TypeIb_triv_for_d_1}
                Let 
                    \begin{align*}
                        \left\{\frac{d}{dt}x = A(\alpha,t)x \right\}_{\alpha \in \mathbb{T}^d}
                    \end{align*}
                be a periodically parameterized periodic lODE and let $T$ be the period of $A(\gamma,\cdot)$. If the monodromy matrix $X_\alpha(T)$ is locally $C^1$-diagonalizable for all $\alpha \in \mathbb{S}^1$, then the associated lODE 
                    $$\left\{\frac{d}{dt}X = A_\alpha(t)X\right\}_{\mathbb{\alpha} \in \mathbb{S}^1\cong \mathbb{R}/L}$$
                has trivial Type I.b Homotopy Class.
            \end{lemma}

            \begin{proof}
                Let
                \begin{align*}
                    \eta = (\eta_1,\ldots,\eta_N): \quad \mathbb{R} \longrightarrow (\mathbb{CP}^{N-1})^N
                \end{align*}
                be as in Definition \ref{def:TypeIb}. That is, let $\eta$ be a continuous lifting of the characteristic multipliers $(\lambda_1,\ldots,\lambda_N)$. Recall that the Type I.b homotopy class associated to $\lambda_n$ was defined as the homotopy classes 
                \begin{align*}
                    \eta_n \in [\mathbb{R}/\tilde{L},\mathbb{CP}^{N-1}] \cong [\mathbb{S}^1,\mathbb{CP}^{N-1}],
                \end{align*}
                with $1 \leq n\leq N$.
                Since complex projective space $\mathbb{CP}^{N-1}$ is simply connected for $N \geq 1$, it follows that $[\mathbb{S}^1,\mathbb{CP}^{N-1}]$ consists of precisely one element and thus the Type I.b Homotopy Class of $\left\{\frac{d}{dt}X = A_\alpha(t)X\right\}_{\alpha \in \mathbb{S}^1}$ is always trivial.
            \end{proof}

            The argument for the Type II.b Homotopy Class is similar. It relies on the fact that $\SU(N)$ is 1- and 2-connected for all $N \geq 1$.

            \begin{lemma}\label{lem:TypeII_triv_for_d_1}
                Let 
                \begin{align*}
                    \left\{\frac{dx}{dt} = A(\alpha,t)x \right\}_{\alpha \in \mathbb{T}^d}
                \end{align*}
            be a periodically parameterized periodic lODE and let $T$ be the period of $A(\gamma,\cdot)$. Assume the setting of Definition \ref{def:TypeIIb} and using the notation of Lemma \ref{lem:homotopy_class_of_Delta_o_P}. Then the Type II.b Homotopy Class of the associated lODE 
                    $$\left\{\frac{dX}{dt} = A_\alpha(t)X\right\}_{\mathbb{\alpha} \in \mathbb{S}^1\cong \mathbb{R}/L}$$
                is uniquely determined by the homotopy class $\Delta\circ P \iota \in \pi^1(\mathbb{S}^1)$.
            \end{lemma}

            \begin{proof}
                Let $P$ be as in Lemma \ref{lem:periodic_Lyapunov_transform} and let $S(\alpha,\overline{t}) \in C^0(\mathbb{R}/L \times \mathbb{R}/T\mathbb{Z}, \SU(N))$ be obtained from $P(\alpha,\overline{t})$ as indicated in equation \eqref{eq:def_of_SU_part}. Recall that the Type II.b Homotopy Class associated to the lODE $\left\{\frac{d}{dt}X = A_\alpha(t)X\right\}_{\alpha \in \mathbb{S}^1}$ is uniquely determined by the homotopy class $\Delta\circ P \iota \in \pi^1(\mathbb{S}^1)$ and the homotopy class of $S$ in $ C^0(\mathbb{R}/L \times \mathbb{R}/T\mathbb{Z}, \SU(N))$. In the case of a one-dimensional parameter space $\mathbb{S}^1$, the space $C^0(\mathbb{R}/L \times \mathbb{R}/T\mathbb{Z}, \SU(N))$ is homeomorphic to $C^0(\mathbb{T}^2, \SU(N))$. However, $\SU(N)$ is 1- and 2-connected for all $n \geq 1$ and thus every continuous function $\mathbb{T}^2 \rightarrow \SU(N)$ is homotopic to a constant function. In particular, $S \in C^0(\mathbb{R}/L \times \mathbb{R}/T\mathbb{Z}, \SU(N))$ is always homotopically trivial. In other words, the  Type II.b Homotopy Class associated to the lODE $\left\{\frac{d}{dt}X = A_\alpha(t)X\right\}_{\mathbb{\alpha} \in \mathbb{S}^1\cong \mathbb{R}/L}$ is uniquely determined by $\Delta\circ P \iota \in \pi^1(\mathbb{S}^1)$.
            \end{proof}

            \subsection{Type I.a homotopic effects}\label{subsec:TypeI.a_top_eff}

            Applying the above theory to the setting of a hexagonal structure  one can observe non-trivial instances of the Type I.a Topological Invariant associated to the corresponding parameterized lODE \eqref{eq:hill}. To the best of our knowledge, non-trivial Type II.b haven't been observed yet, indicating that the Type I.a invariant fully characterises the homotopy class of subwavelength solutions in high-contrast time-dependent metamaterials.

            \begin{figure}[h]
                \centering
                \includegraphics{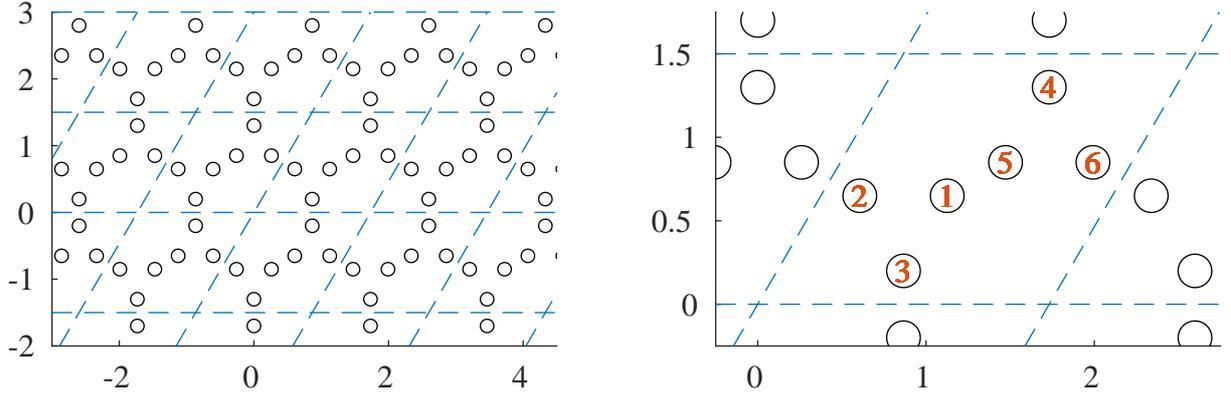}
                \caption{Displayed is the hexagonal structure used in subsection \ref{subsec:TypeI.a_top_eff}. The lattice is displayed with blue dashed lines and the resonators are displayed as black outlined circles. In the right figure, the resonators inside the fundamental domain $\{a_1g_1 + a_2g_2 | a_1, a_2 \in [0,1) \}$ are displayed and numbered according to the notation used in subsection \ref{subsec:TypeI.a_top_eff}. One can distinguish the two trimers: resonators 1, 2, 3 and resonators 4, 5, 6.  }
                \label{fig:hexagonal_structure_metamaterial}
            \end{figure}
                        
            Two structures of non-trivial homotopy type that are of nontrivial Type I.a will be presented in the following section. The \emph{static} structure will in both cases be given by the same structure depicted in Figure \ref{fig:hexagonal_structure_metamaterial}. 

            It is given by six circular resonators $D_1, \ldots, D_6$ in the fundamental domain $\{a_1g_1 + a_2g_2\, |\, a_1, a_2 \in [0,1)\}$ associated to the lattice vectors 
            \begin{align*}
                g_1 = \begin{pmatrix} \sqrt{3} \\
                    0 \end{pmatrix}, \quad g_2 = \begin{pmatrix} \sqrt{3}/2 \\
                        3/2 \end{pmatrix}.
            \end{align*}

            \begin{figure}[h]
                \centering
                \includegraphics{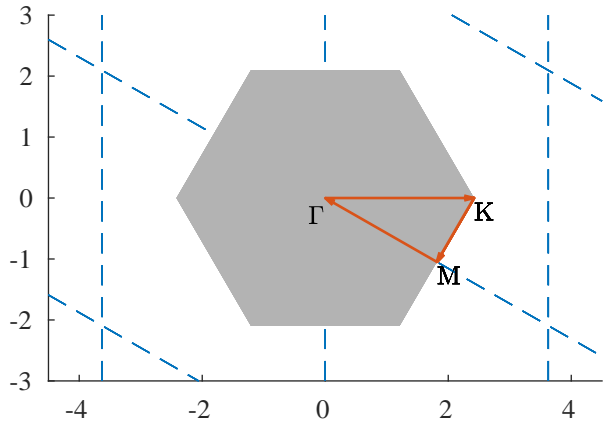}
                \caption{The first Brillouin zone associated to the hexagonal structure used in subsection \ref{subsec:TypeI.a_top_eff} is colored in grey. On it are displayed the high symmetry points ($~\Gamma$, $K$ and $M$) and the (high) symmetry curve, which is indicated with orange arrows. The reciprocal lattice is displayed using blue dashed lines.}
                \label{fig:brioullin_zone}
            \end{figure}

            The resonators have equal radius $R = 0.1$ and are arranged in two trimer blocks (resonators 1, 2, 3 and resonators 4, 5, 6, respectively). The positions of the centers $c_1, \ldots, c_6$ of the six resonators $D_1,\ldots, D_6$ in the fundamental domain are give by

            \begin{align*}
                c_1 = \frac{1}{3}(g_1 + g_2) + 3R\begin{pmatrix} \cos(\frac{\pi}{6}) \\ \sin(\frac{\pi}{6}) \end{pmatrix}, ~
                c_2 = \frac{1}{3}(g_1 + g_2) + 3R\begin{pmatrix} \cos(\frac{5\pi}{6}) \\ \sin(\frac{5\pi}{6}) \end{pmatrix}, ~
                c_3 = \frac{1}{3}(g_1 + g_2) + 3R\begin{pmatrix} \cos(\frac{9\pi}{6}) \\ \sin(\frac{9\pi}{6}) \end{pmatrix},\\
                c_4 = \frac{2}{3}(g_1 + g_2) - 3R\begin{pmatrix} \cos(\frac{9\pi}{6}) \\ \sin(\frac{9\pi}{6}) \end{pmatrix}, ~
                c_5 = \frac{2}{3}(g_1 + g_2) - 3R\begin{pmatrix} \cos(\frac{\pi}{6}) \\ \sin(\frac{\pi}{6}) \end{pmatrix}, ~
                c_6 = \frac{2}{3}(g_1 + g_2) - 3R\begin{pmatrix} \cos(\frac{5\pi}{6}) \\ \sin(\frac{5\pi}{6}) \end{pmatrix}.
            \end{align*}

            In the case of such a highly symmetric static structure it suffices to consider the capacitance matrix formulation on the \emph{(high) symmetry curve} of the reciprocal lattice $L$. This reduces the a priori 2-dimensionally parameterized system to a 1-dimensionally parameterized system. The new parameterization set is then given by the piecewise linearly interpolated path through the points $\Gamma$, $K$ and $M$. It is depicted in Figure \ref{fig:brioullin_zone}.

            \begin{figure}[h]
                \centering
                \begin{subfigure}[b]{0.45\textwidth}
                    \centering
                    \includegraphics[width=\textwidth]{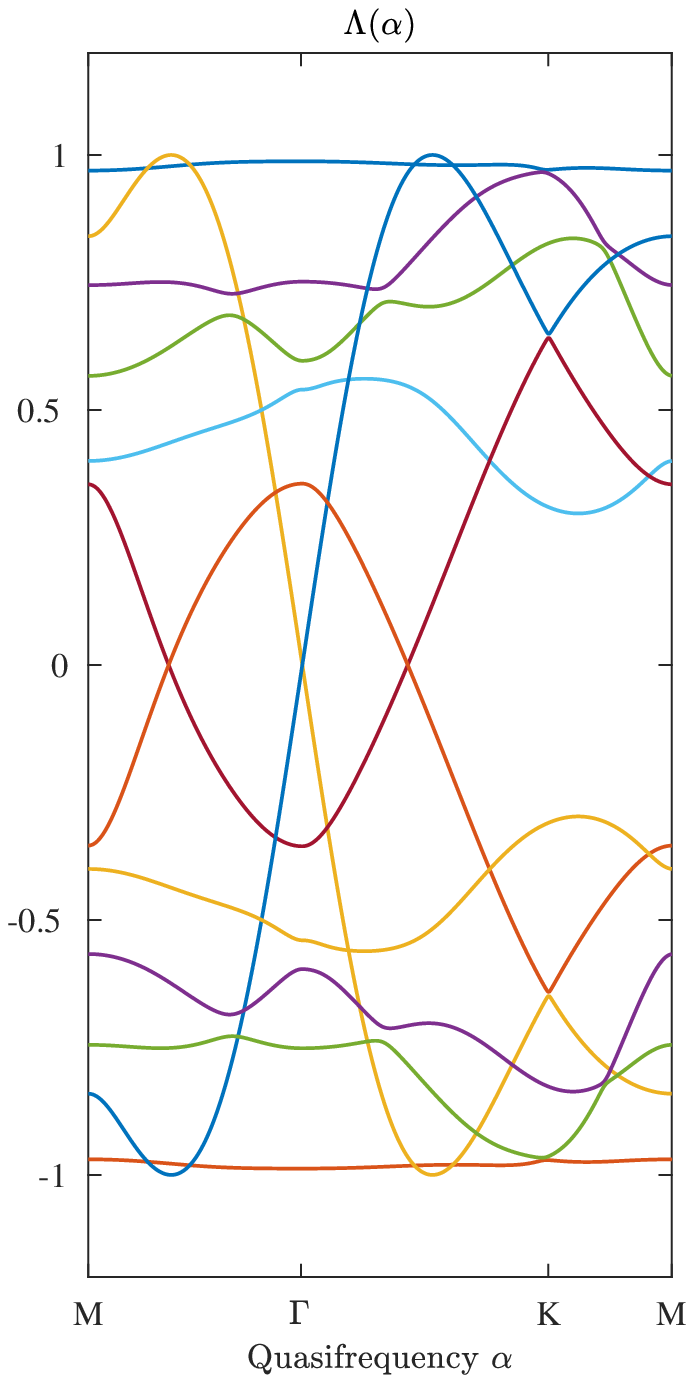}
                    \caption{Band structure on the symmetry curve.}
                    \label{fig:hexagonal_structure_metamaterial_N_100}
                \end{subfigure}
                \hfill
                \begin{subfigure}[b]{0.45\textwidth}
                    \centering
                    \includegraphics[width=\textwidth]{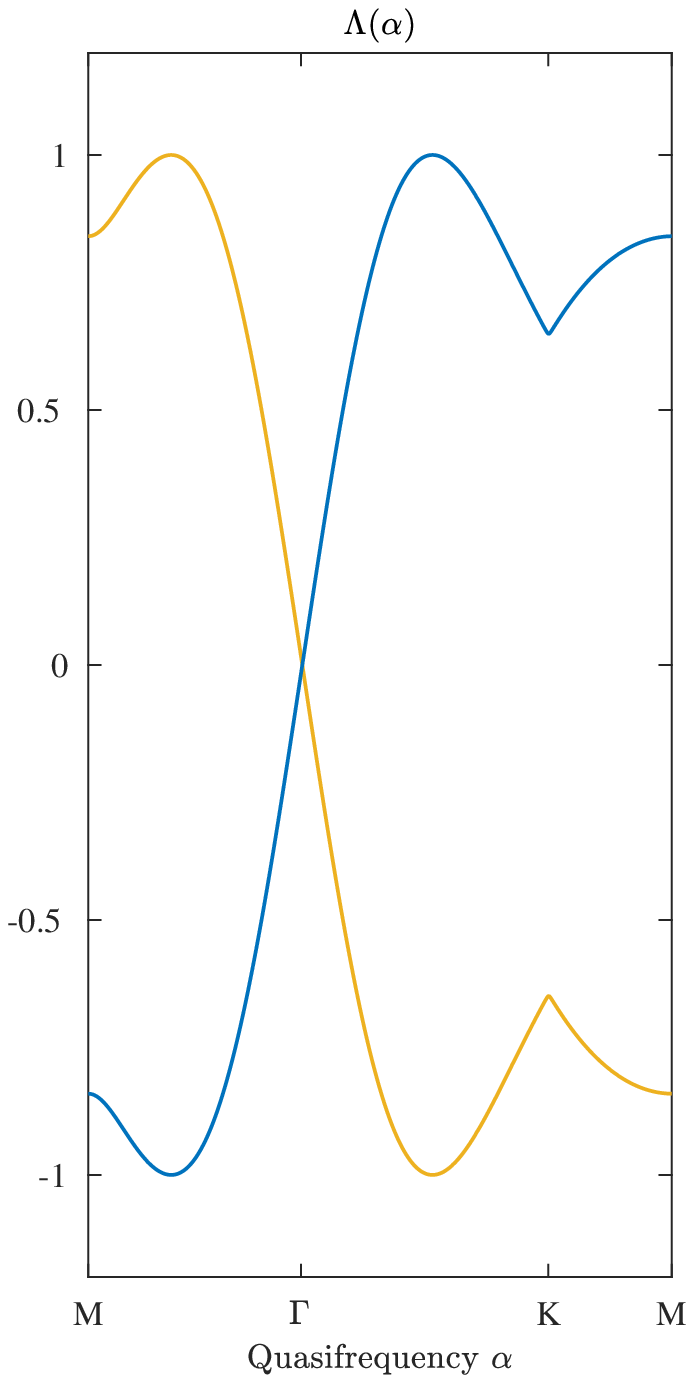}
                    \caption{Topological non-trivial bands.}
                    \label{fig:hexagonal_structure_metamaterial_N_100_winding_band}
                \end{subfigure}
                \caption{Example of a band structure with non-trivial Type I.a invariant. The band structure of the hexagonal structure presented in this section with modulation given by equation \eqref{eq:modulation1} and step size equal to $1.21 \times 10^{-2}$, which means $100$ steps between $M$ and $\Gamma$. }
            \end{figure}

            By the results of the previous subsection \ref{subsec:S1_param_top_invar}, it thus follows that the Type I.b and and the homotopy class of $S(\alpha,\overline{t})$ are always trivial in this case. Hence, the Type I.a Topological Invariant is a good indicator for the different topological nature of subwavelength solutions to the wave equation associated to a time-modulated hexagonal structure. It remains to show that this invariant takes interesting values for some instances of time-modulated hexagonal structure. In the following we will thus present two examples of different instances of the Type I.a Topological Invariant.
            
            The following modulation of the bulk modulus $\kappa$ inside the resonators of the above defined material gives an example of a modulated hexagonal structure where the associated Type I.a Topological Invariant is \emph{non-trivial} and \emph{equal to 2}.

            Using the notation from equation \eqref{eq:resonatormod}, the following modulation was used 
            \begin{align}\label{eq:modulation1}
                \kappa_1(t) = 1/(1+\epsilon\cos(\Omega t)),~ \kappa_2(t) = 1/(1+\epsilon\cos(\Omega t+2\pi/3)),~ \kappa_3(t) = 1/(1+\epsilon\cos(\Omega t+4\pi/3)),~\\
                \kappa_4(t) = 1/(1+\epsilon\cos(\Omega t)),~ \kappa_5(t) = 1/(1+\epsilon\cos(\Omega t+2\pi/3)),~ \kappa_6(t) = 1/(1+\epsilon\cos(\Omega t+4\pi/3)),~ 
            \end{align}
            where $\Omega = 0.2$ and $\epsilon = 0.001$. The density $\rho$ was not modulated and set to $\rho_0 = 1$ for the background material and equal to $1/9000$ inside the resonators, giving a contrast parameter of $\delta = 1/9000$, ensuring the high-contrast regime.
            \begin{figure}[h!]
                \centering
                \begin{subfigure}[b]{0.32\textwidth}
                    \centering
                    \includegraphics[width=\textwidth]{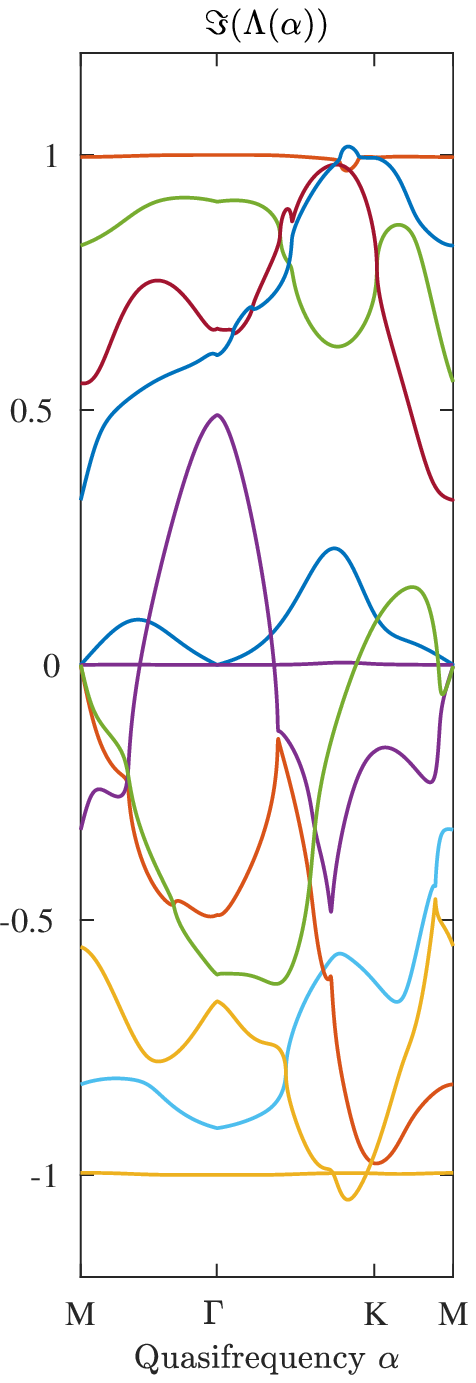}
                    \caption{Imaginary part of band structure on the symmetry curve.}
                    \label{fig:hexagonal_structure_metamaterial_N_200_strong_modulation}
                \end{subfigure}
                \hfill
                \begin{subfigure}[b]{0.32\textwidth}
                    \centering
                    \includegraphics[width=\textwidth]{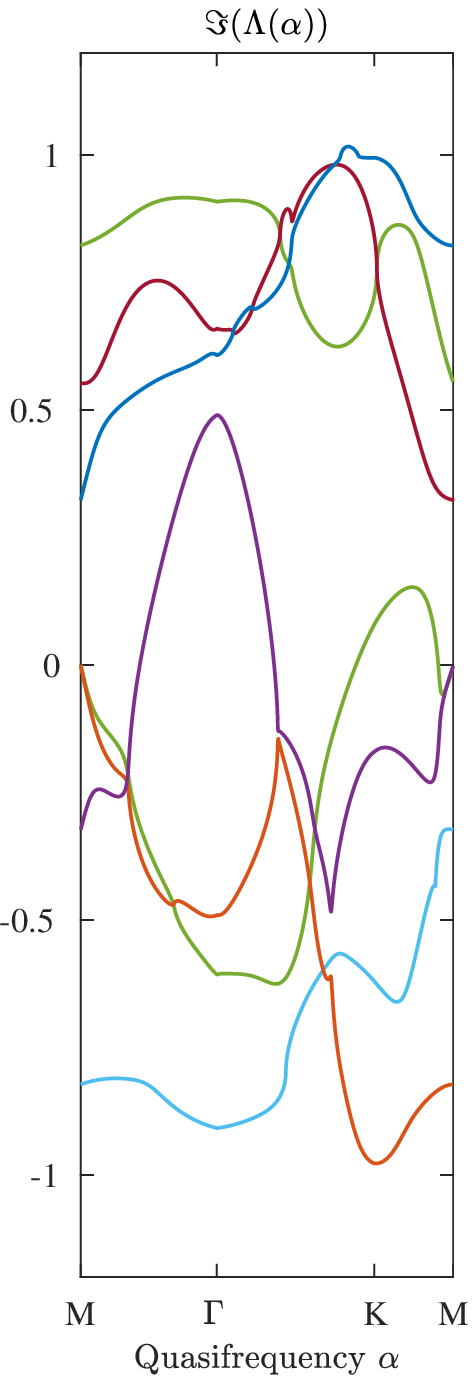}
                    \caption{Imaginary part of topological non-trivial bands.}
                    \label{fig:hexagonal_structure_metamaterial_N_200_strong_modulation_nontriv_bands_1}
                \end{subfigure}
                \hfill
                \begin{subfigure}[b]{0.32\textwidth}
                    \centering
                    \includegraphics[width=\textwidth]{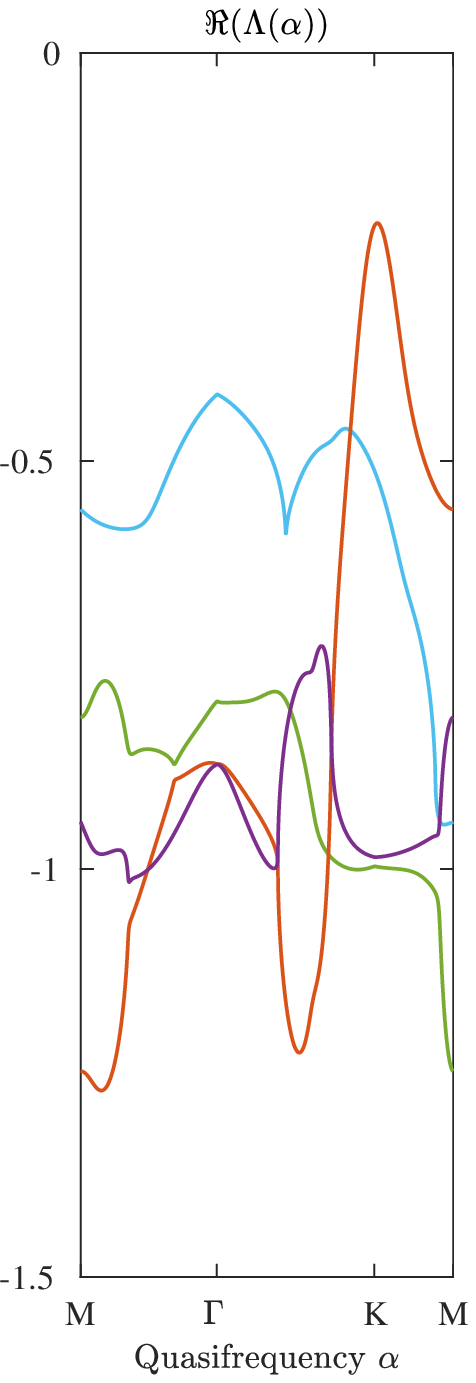}
                    \caption{Real part of lower block of topological non-trivial bands.}
                    \label{fig:hexagonal_structure_metamaterial_N_200_strong_modulation_nontriv_bands_2}
                \end{subfigure}
                \caption{Example of the band structure of a strong modulation with non-trivial Type I.a invariant. The band structure of the time-modulated hexagonal structure with modulation given by equations \eqref{eq:modulation1_kappa}-\eqref{eq:modulation1_rho}. The step size was set to $6.05 \times 10^{-3}$, which corresponds to 200 steps between $M$ and $\Gamma$. On Figure \ref{fig:hexagonal_structure_metamaterial_N_200_strong_modulation}, the imaginary part of all bands are displayed. In Figure \ref{fig:hexagonal_structure_metamaterial_N_200_strong_modulation_nontriv_bands_1} only the imaginary part of those bands which are topologically non-trivial are depicted. This is not evident in the case of the green band. For this reason also the real part of the lower group of topologically non-trivial bands is displayed in Figure \ref{fig:hexagonal_structure_metamaterial_N_200_strong_modulation_nontriv_bands_2}.}
            \end{figure}

            This setting has the associated band structure depicted in Figure \ref{fig:hexagonal_structure_metamaterial_N_100}. Most bands can be glued continuously across the point $M$. However, one band does not connect continuously (see Figure \ref{fig:hexagonal_structure_metamaterial_N_100_winding_band}), but connects with another band which then connects to the former. Thus, the Type I.a Topological Invariant is given by $2$ and the structure is topologically non-trivial.
            
            Next, we will present another example of a time-modulation of the hexagonal structure  displayed in Figure \ref{fig:hexagonal_structure_metamaterial}. This example has the same periodicity in space as well as in time (i.e., the same frequency) as the first example. However, it has three main differences. First, not only the bulk modulus $\kappa$ but also the density $\rho$ is modulated within the resonators, second, a considerably stronger modulation is used. When in the former example a modulation amplitude of $\epsilon = 0.001$ was used, in the following example the modulation will be 300-500 times larger, having an amplitude of $\epsilon_\kappa = 0.5$ for the bulk modulus $\kappa$ and an amplitude of $\epsilon_\rho = 0.3$ for the density $\rho$. Third, the phase shifts of the modulations of the resonators are different.
            
            Using the notation from equation \eqref{eq:resonatormod}, the following modulation of the bulk modulus $\kappa$ was used: 
            \begin{align}\label{eq:modulation1_kappa}
                \kappa_1(t) = \frac{1}{1+\epsilon_\kappa\cos(\Omega t)},\quad \kappa_2(t) = \frac{1}{1+\epsilon_\kappa\cos\left(\Omega t+\frac{2\pi}{3}\right)},\quad \kappa_3(t) = \frac{1}{1+\epsilon_\kappa\cos\left(\Omega t+\frac{4\pi}{3}\right)},\\
                \kappa_4(t) = \frac{1}{1+\epsilon_\kappa\cos\left(\Omega t+\frac{2\pi}{3}\right)},\quad \kappa_5(t) = \frac{1}{1+\epsilon_\kappa\cos(\Omega t)},\quad \kappa_6(t) = \frac{1}{1+\epsilon_\kappa\cos\left(\Omega t+\frac{4\pi}{3}\right)},
            \end{align}
            where $\Omega = 0.2$ and $\epsilon_\kappa = 0.5$. The shape of the modulation is different from the first example, for this example, it is \emph{mirror-symmetric} between the two trimers.
            
            Defining the modulation of the density, using the notation from equation \eqref{eq:resonatormod}, the following modulation of the density $\rho$ was used 
            \begin{align}
                \rho_1(t) = \frac{1}{1-\epsilon_\rho\cos(\Omega t)},\quad \rho_2(t) = \frac{1}{1-\epsilon_\rho\cos(\Omega t+2\pi/3)},\quad \rho_3(t) = \frac{1}{1-\epsilon_\rho\cos(\Omega t+4\pi/3)}, \\
                \rho_4(t) = \frac{1}{1-\epsilon_\rho\cos(\Omega t+2\pi/3)},\quad \rho_5(t) = \frac{1}{1-\epsilon_\rho\cos(\Omega t)},\quad \rho_6(t) = \frac{1}{1-\epsilon_\rho\cos(\Omega t+4\pi/3)},\label{eq:modulation1_rho}
            \end{align}
            where $\epsilon_\rho = 0.3$. Note, that also $\rho$ was chosen to be mirror-symmetric between the two trimers.            

            This choice of modulation leads to the band structure depicted in Figure \ref{fig:hexagonal_structure_metamaterial_N_200_strong_modulation}. Its instance of the Type I.a Topological Invariant is non-trivial, in fact, it is given by $12$. In Figures \ref{fig:hexagonal_structure_metamaterial_N_200_strong_modulation_nontriv_bands_1} the two groups of non-trivial bands are displayed. The first group (see the top part of Figure \ref{fig:hexagonal_structure_metamaterial_N_200_strong_modulation_nontriv_bands_1}) being of multiplicity $3$ and the second group (see the lower part of Figures \ref{fig:hexagonal_structure_metamaterial_N_200_strong_modulation_nontriv_bands_1} and \ref{fig:hexagonal_structure_metamaterial_N_200_strong_modulation_nontriv_bands_2}) is of multiplicity $4$, leading to a Type I.a Topological Invariant of $12$. In order to identify the multiplicity of the second group it is necessary to consider the real part of the bands, which are displayed in Figure \ref{fig:hexagonal_structure_metamaterial_N_200_strong_modulation_nontriv_bands_2}. There, it becomes apparent that the purple band is continued by the green band which in turn is continued by the orange band.

    \section{Concluding remarks}
        In this paper, a through topological characterization of the time-local and time-global behaviors of the fundamental solutions of periodically parameterized periodic lODEs was introduced (see section \ref{sec:Topological_phenomena_of_periodically_parameterized_linear_ODEs}). The main starting point of the theory was given by a generalisation of the Floquet normal form of fundamental solutions of periodic lODEs. In fact, a Floquet normal form was introduced, which depends continuously and periodically on the parameter of the considered periodic lODE, allowing for the consideration of homotopy classes of the respective components of the Floquet normal form. Having introduced the respective topological invariants: Type I.a Topological Invariant and Type I.b Homotopy Class for the time-global behavior and Type II.a and Type II.b Homotopy Class for the time-local behavior of the associated parameterized fundamental solution, we were able to perform some analysis on the complexity of the system (e.g. degree of the lODE and the dimension of the parameter space) needed in order to obtain certain topological effects, see subsection \ref{subsec:S1_param_top_invar}. Most notably, non-trivial Type I.b Homotopy Classes can only be obtained in the case where the parameter space $\mathbb{T}^d$ is of dimension $d \geq 2$. The Type I.a Topological Invariant and the homotopy class of $\Delta \circ P \iota \in \pi^1(\mathbb{S}^1)$ being meaningful in the greatest amount of settings, one can use them to topologically distinguish systems with parameter spaces of any dimension $d\geq 1$ and the former is even applicable in the setting of lODEs with constant coefficients.

        The implications of this new topological theory for periodically time-modulated hexagonal structures in the high contrast, subwavelength regime are the following (see section \ref{sec:Application_to_high-contrast_acoustic_hexagonal_metamaterial_structures}). Due to the symmetry of a hexagonal structure, some topological effects of the associated subwavelength solutions to the time-modulated material can be distinguished by the Type I.a topological invariant.

        We were able to provide two interesting modulation examples of a hexagonal structure which were of non-trivial Type I.a, proving that topologically non-trivial modulations of hexagonal structures do exist. We even provided an example of a relatively high instance of the Type I.a Topological Invariant, the example provided had was brading 12 times.

        To come back to the original motivation of this paper, namely the analogue of the bulk boundary correspondence in the setting of Floquet metamaterials, a next step would be to built upon the results of this paper and analyse the provided examples for the occurrence of edge modes.
        
        Furthermore, we would like to point out that the developed theory can be refined by considering real-valued coefficient matrices $A: \mathbb{R}^d \times \mathbb{R} \rightarrow \Mat_{N \times N}(\mathbb{R})$ instead of complex valued coefficient matrices. This would particularly enrich the Type II.b topological effects. Indeed, it would mean that the Homotopy Class of $S(\alpha, \overline{t})$ would no longer be taken in the space $C^0(\mathbb{T}^{d+1}, \SU(N))$ but in the space $C^0(\mathbb{T}^{d+1}, \SO(N))$, see e.g. Remark \ref{rem:theory_real_coef_mat}. This might provide more topological effects, since $\SO(N)$ has richer topological properties than $\SU(N)$, in particular $\SO(N)$ is not simply connected. This would in particular imply that the \emph{real} Type II.b Homotopy Class also depends on the Homotopy Class of $S(\alpha, \overline{t})$, even in the case of a one-dimensional parameter space.

\bibliographystyle{abbrv}
\bibliography{references}

\end{document}